\documentclass[12pt, reqno]{amsart}
\usepackage{amsfonts}
\usepackage{bbm}
\usepackage{amscd,amsfonts}
\usepackage{amssymb, eucal, amsfonts, amsmath, xypic, latexsym}
\usepackage{pifont}
\usepackage{mathrsfs,color}
\usepackage{amsthm,indentfirst,bm,fancyhdr,dsfont}
\usepackage{graphicx}
\usepackage[all]{xy}
\usepackage[CJKbookmarks=true]{hyperref}

\usepackage{mathrsfs}
\usepackage{amsmath}
\usepackage{amssymb}
\usepackage{hyperref}

\newtheorem{theorem}{Theorem}[section]
\newtheorem{lemma}[theorem]{Lemma}
\newtheorem{prop}[theorem]{Proposition}

\newtheorem{corollary}[theorem]{Corollary}
\theoremstyle{definition}

\newtheorem{defn}[theorem]{Definition}

\newtheorem{remark}[theorem]{Remark}

\newtheorem{convention}[theorem]{Convention}

\numberwithin{equation}{section}

\def\ggg{\mathfrak{g}}

\def\co{\mathcal{O}}

\def\calh{\mathcal{H}}
\def\calk{\mathcal{K}}

\def\ggg{\mathfrak{g}}

\def\ppp{\mathfrak{p}}

\def\hhh{\mathfrak{h}}

\def\bbb{\mathfrak{b}}
\def\fff{\mathfrak{f}}

\def\bbc{\mathbb{C}}
\def\bbf{\mathbb{F}}
\def\bbz{\mathbb{Z}}

\def\bk{\mathbf{k}}
\def\bbn{\mathbb{N}}

\def\bbk{{\mathbb{K}}}
\def\tbk{\textbf{k}}

\def\bc{{\mathbf{c}}}

\def\bo{{\bar 1}}
\def\bz{{\bar 0}}
\def\ev{{\text{ev}}}

\def\sfF{\textsf{F}}

\def\scrO{{\mathscr{O}}}
\def\scrL{{\mathscr{L}}}
\def\scrI{{\mathscr{I}}}

\def\scrW{\mathscr{W}}

\def\bwedge{{\bigwedge\hskip-2pt{}^{{}^{\bullet}}}}

\def\Lie{\text{Lie}}

\def\GL{\text{GL}}
\def\Hom{\text{Hom}}

\def\hmod{\text{-\bf{mod}}}
\def\id{\mathsf{id}}

\def\ind{\texttt{ind}}

\def\total{ H_{\textsf{total}}^0(\lambda)  }

\def\bmn{{B^{(mn)}}}

\def\sak{{\texttt{salg}_\tbk}}

{\vskip-\lastskip\medskip
  \noindent
  {\em #1.}\enspace
  }%
{\qed\par\medskip
  }
\begin{document}

\title[Jantzen filtration of Weyl modules]
{Jantzen filtration of Weyl modules for general linear supergroups}
\author{Yi-Yang Li and Bin Shu}
\address{School of Mathematics, Physics and Statistics, Shanghai University of Engineering Science,
Shanghai 201620, China}\email{yiyangli1979@outlook.com}
\address{Department of Mathematical Sciences, East China Normal University,
Shanghai 200241,  China} \email{bshu@math.ecnu.edu.cn}

\thanks{This work is supported by the National Natural Science Foundation of China (Grant No. 12071136, 11771279 and 12271345), and by Science and Technology Commission of Shanghai Municipality (No. 22DZ2229014).}

\begin{abstract} Let $G=\GL(m|n)$ be a general linear supergroup over  an algebraically closed field $\tbk$ of odd characteristic $p$.
In this paper we construct  Jantzen filtration of Weyl modules $V(\lambda)$ for $G$ when $\lambda$ is a typical weight in the sense of Kac's definition, and consequently obtain a sum formula for their characters. By Steinberg's tensor product theorem, it is enough for us to study typical weights with aim  to formulate irreducible characters.
As an application, it turns out that an irreducible $G$-module $L(\lambda)$ can be realized as a Kac module if and only if $\lambda$ is $p$-typical.
\end{abstract}

\maketitle
\setcounter{tocdepth}{1}\tableofcontents
\section*{Introduction}

\subsection{} As is well-known, it is a very important but very difficult problem to formulate irreducible characters for reductive algebraic groups in prime characteristic (see  \cite{AJS}, \cite[\S{II.C}]{Jan3}, \cite{Wil}, etc.). Establishing Jantzten filtration and its sum formula of characters for Weyl modules is an approach to understanding the question (see \cite[\S{II.8}]{Jan3}). As a counterpart, it's a significant task to establish Jantzen filtration in the study of representations of algebraic supergroups.
In this paper, we initiate to study Jantzen filtration for algebraic supergroups, beginning with the case of the general linear supergroup $GL(m|n)$.

\subsection{} Let $\textbf k$ be an given  algebraically closed field of characteristic $p>2$. Let $G$ be an algebraic supergroup of Chevalley type  over  $\textbf{k}$ in the sense of \cite{FG12}, $G_{\ev}$ is the purely-even subgroup of $G$, $T$ be a maximal torus of $G$ with $X(T)$ being the character group, and $B$ the Borel subgroup corresponding to negative roots. Denote by $W$ the Weyl group of $G$,  by  $w_{\bar0}$ the longest elements in $W$ and $l(w_{\bar 0})$ is the length of $w_{\bar 0}$. Let $\ggg=\texttt{Lie}(G)$ and $\bbb^-=\texttt{Lie}(B)$.

 \subsection{} When considering a connected reductive algebraic group $G_\ev$ (here we abuse some notations like  $G_{\ev}$ with subscript $\ev$ for the time being to void any confusions), one has ``standard modules" ($H_\ev^0(\lambda):=H^0(G_\ev/B_\ev, \scrL_{G_\ev/B_\ev}(\textbf k_\lambda))$) and ``costandard modules" ($V(\lambda):=H^0_\ev(-w_{\bar0}\lambda)^{*}$) with $\lambda\in X^+(T)$ (the set of dominant weights). The  Weyl module $V_{\ev}(\lambda)$ is isomorphic to $H_{ev}^{l(w_{\bar0})}(w_{\bar0}.\lambda)$ under the Serre duality (cf. \cite[\S{II.4.2}]{Jan3}). The character formulas for $V_\ev(\lambda)$ and  $H_\ev^0(\lambda)$ are the same.
   The theory of Jantzen filtrations and the sum formulas for  Weyl modules plays a very important role in representations of $G_\ev$ (cf. \cite[\S{II.8}]{Jan3}).

\subsection{}\label{sec: totally early} The standard modules $H^0(\lambda)$ for algebraic supergroup $G$  have been  studied in different versions  (see \cite{BK03, B06, MZ1, Shi,  Z2, Z1}).
%
%
In this paper, we will focus on   $G=GL(m|n)$  unless other stated.
Define the  Weyl supermodule $V(\lambda):=H^{l(w_{\bar0})}(w_{\bar0}.\lambda)$, the Serre duality for  superschemes (see Theorem \ref{Serre duality}) makes $V(\lambda)$ isomorphic to $H^{0}(G/B, \textbf k_{-w_{\bar 0}\lambda}\otimes \bwedge(\ggg\slash \bbb^-)_\bo)^*$ (see (\ref{f8.1})).
Although not as satisfactory with $V(\lambda)$ as in the case of reductive algebraic groups,  it turns out that we have the following axioms  (see Lemma \ref{lem9.1}).
\begin{itemize}
\item[(W1)] The socle of $H^{0}(G/B, \textbf k_{-w_{\bar 0}\lambda}\otimes \bwedge(\ggg\slash \bbb^-)_\bo)$ is isomorphic to $L(-w_{\bar 0}\lambda+2\rho_{\bar 1})$) where $\rho_{\bar 1}$ is half the sum of positive odd roots.
    \item [(W2)] Correspondingly, the head of $V(\lambda)$ is isomorphic to $L(-w_{\bar 0}\lambda+2\rho_{\bar 1})^*$.
        \item[(W3)] The following character formula  holds $$\texttt{ch}(V(\lambda))=\texttt{ch}(H^0(\lambda)).$$
            \end{itemize}
So, the Weyl modules are still important as in the ordinary case although it becomes more complicated.

Indeed,  different from the case of reductive algebraic groups, the head of $V(\lambda)$ is  no more isomorphic to the socle of $H^0(\lambda)$ in general. There is no longer enough information from $H^0(\lambda)$ providing for the construction of Jantzen filtration because only considering the Weyl group $W$ it is not enough for us to take the whole picture, due to the emergence of odd parts.

 So, in order to establish  the Jantzen filtration of $V(\lambda)$, we need to consider the influence caused by the odd reflections. Especially, we need to introduce a so-called totally-odd induced module $H^{0}(G/w_\bo(B), \scrL_{G/w_\bo(B)}(\textbf k_{\lambda-2\rho_\bo}))$ with respect to the Borel subgroup $w_\bo(B)$ defined by total negative odd roots where $\rho_\bo$ stands for half the sum of positive odd roots, and $w_\bo$ is the ``longest" element of the part  generated odd reflections of the super Weyl group $\widehat W$ introduced in \cite{PS} for construction of Jantzen filtration in modular representations of basic classical Lie superalgebras. We will denote this totally-odd induced module by $H_{\textsf{total}}^0(\lambda)$.  In our construction, $H_{\textsf{total}}^0(\lambda)$ will take the place of $H^0(\lambda)$ in the situation of reductive algebraic groups, which turns out to satisfy the following axiom (see Lemma \ref{lem: totally ind ch}).
 \begin{itemize}
 \item[(W4)] $\texttt{ch}(\total)=\texttt{ch}(H^0_\ev(\lambda-2\rho_\bo))\Xi_{mn}$ with $\Xi_{mn}=\prod_{\beta\in \Phi^+_\bo}(1+e^\beta)$.  
     \end{itemize}

\subsection{} Now we make a  rough explanation on  the above.
In contrast with reductive algebraic groups,
  the critical difference in the super case is that there exist non-conjugate Borel subgroups even under the circumstance that they contain the same maximal torus. Correspondingly, the role of Weyl groups of purely-even subgroups is not enough powerful in describing Weyl modules. We need to add the role of odd reflections integrated into the so-called super Weyl groups. Then, in some sense, for example, under the action of super Weyl groups (see  \S\ref{2.2.2} or \cite[\S3.1]{PS} for more details),  all Borel subgroups  containing the same maximal torus are conjugate.
   So it enables us to have the following pictures involving $V(\lambda)$. Take $T$ to be the standard maximal torus of $\GL(m|n)$ and denote by $\Phi$ the root system associated with $T$.
    One can make an order of the positive odd roots with respect to {the  Borel subgroup $B^+$ with $B^+(R)$ consisting of all  upper-triangle matrices in $\GL(m|n)(R)$ for $R\in\sak$},
    as $\{\beta_1,\beta_2,\ldots,\beta_{mn}\}$ (note that all odd roots  for $\GL(m|n)$ are isotropic). Denote by $\hat r_{\beta_i}$ the corresponding odd reflections (see \S\ref{2.2.2}). Then a sequence of Borel subgroups $\{B^{(i)}\mid i=1,\ldots,mn\}$ arise.
  Naturally,  with respect to  $w_\bo(B)=B^{(mn)}$ we already have talked about the totally-odd induced module $H_{\textsf{total}}^0(\lambda)=H^{0}(G/w_\bo(B), \scrL_{G/ w_\bo(B)}(\textbf k_{\lambda-2\rho_\bo}))$. The socle of $H_{\textsf{total}}^0(\lambda)$ turns out isomorphic to  $L(-w_{\bar 0}\lambda+2\rho_{\bar 1})^*$  which coincides with the head of $V(\lambda)$, up to isomorphisms (see \S\ref{S8.2}).

Next, it is revealed that there exists connection $V(\lambda)$ with $H_{\textsf{total}}^0(\lambda)$.
 We construct a series of homomorphisms arising from an ordered  odd reflections along with the ordinary longest element ${w}_\bz$ in $W$ (see Lemma \ref{lem: longest Weyl ele}).
 Composing the odd parts, we first establish  the nonzero  homomorphism
 $$\widetilde T_{A,w_{\bar1}}: H^0_A(\lambda)\rightarrow H^{0}(G_A/w_\bo(B)_A, \scrL_{G/w_\bo(B)}(A_{\lambda-2\rho_\bo}))$$
  for any commutative $\bbz$-algebra $A$ (see \S \ref{7.5}).
  For the other part arising from $w_\bz$, one can  mimic the arguments for reductive algebraic groups.
In summary, the construction is   based on the reduced expression of the longest element $\widehat{w}_\ell=w_{\bar 0}w_{\bar 1}$ of the super Weyl group $\widehat W$.
 So our arguments on the chain of homomorphisms mentioned above are divided into two parts. One part is from real reflections, which is actually the same as in \cite[\S{II.8.16}]{Jan3},  giving rise to a homomorphism  $\widetilde T_{A,w_{\bar0}}: V_A(\lambda)\rightarrow H^0_A(\lambda)$.  The other part are completely the new phenomenon, from odd reflections, which reflects  the cohomological information arising from the  transition of Borels  by odd refections (see \S\ref{sec: 8.3}).
 The finial composite
  \begin{align*}\widetilde T_{A,\widehat w_{\ell}}=&\widetilde T_{A, {w}_\bo}\circ \widetilde T_{A, w_{\bar 0}}:\cr
  & H^{l(w_{\bar 0})}
 (G_A/B_A, \scrL_{G_A/ B_A}(A_{w_{\bar 0}.\lambda}))\rightarrow  H^0(G_A/w_\bo(B)_A, \scrL_{G_A/ w_\bo(B)}(A_{\lambda-2\rho_\bo}))
 \end{align*}
  turns out to be nonzero when $\lambda$ is typical\footnote{In the present paper, the notion ``typical weight" is identical to Kac's definition over complex numbers in \cite{Kac} (see Definition \ref{def: typical}), different from the one in \cite{ZS} and \cite{Z1}, the latter of which is identical another notion ``$p$-typical weight" here (see Remark \ref{rem: p-typical}).}. In particular, when taking $A$ to be a principal ideal domain such that the fractional field $\bbk$ has characteristic $0$, then $\widetilde T_{\bbk,\widehat w_{\ell}}$ is an isomorphism for typical $\lambda$. By localization of $A$ at a special prime ideal, one has that $\widetilde T_{\textbf k,\widehat w_{\ell}}$  maps the head of $V(\lambda)$ to the socle of $\total$ (see Proposition \ref{c9.2}).

\subsection{} Based on the above results,  we successfully establish  the Jantzen filtration of the Weyl module $V(\lambda)$ and the sum formula for characters when $\lambda$ is typical (see Theorem \ref{main thm}). As to atypical weights, Steinberg's tensor theorem enables us to reduce the question of formulating irreducible characters to the case of  typical weights (see Theorem \ref{thm: st ten prod}).  As an application,  we investigate the question of realizing irreducible modules via  Kac modules, the latter of which are induced modules from irreducible $G_\ev$-module $L_\ev(\lambda)$, regarded $B^+G_\ev$-modules (see \S\ref{sec: kac mod}). It is deduced that an irreducible module $L(\lambda)$ is isomorphic to a Kac module if and only if $\lambda$ is a $p$-typical, i.e. $(\lambda+\rho,\beta)\not\equiv0\mod p$ for all $\beta\in \Phi^+_\bo$ (see Theorem \ref{thm: p-typical}). This is a modular version of typical irreducible modules over complex numbers (see Remarks \ref{rem: complex typical} and \ref{rem: p typical realization}).

\subsection{}The paper is organized as follows. In Section 1, we list the notations and some preliminary results. In Section 2, we present  the induction and restriction functors on the quotients of supergroup schemes and associated sheaves  in the spirit of \cite[\S2]{B06}, in particular, some material is devoted to the super version of Serre duality. In Sections 3 and 4,  on the basis of computation on induced modules of $\GL(1|1)$, we study  homomorphisms between two induced modules  arising from two Borel subgroups which are adjacent by  odd reflections. Here  the arguments on $\GL(1|1)$ in \cite{MZ, Z1} are much used.  In Section 5, we introduce typical weights and Steinberg's tensor product theorem, showing that it is enough for us to understand the case of typical weights (see Theorem \ref{thm: st ten prod}).
In Section 6, we show  a nonzero homomorphism from the induced module $H^0(G\slash B, \lambda)$ to the totally-odd induced module $H^0(G\slash w_\bo(B), \lambda-2\rho_\bo)$. In Section 7, with the help of Serre duality we analysis the  Weyl module $V(\lambda)$.
 In Section 8,
 we establish  the Jantzen filtration of $V(\lambda)$ and a sum formula (Theorem \ref{main thm}). In the concluding section, we introduce Kac modules, and obtain a result on Kac module realization of $p$-typical irreducible modules (Theorem \ref{thm: p-typical}).
\section{Preliminaries}
\subsection{Basic notations and conventions}\label{sec: conv}
Throughout the paper (particularly from Section \ref{sec: third} on), the notions of vector spaces (resp. modules, subgroups) means vector superspaces (resp. super-modules, super-subgroups). For simplicity, we often omit the adjunct word ``super."  Preliminarily, by a commutative superalgebra $R=R_\bz\oplus R_\bo$ it means that  $ab=(-1)^{|a||b|}ba$ for any $\bbz_2$-homogeneous elements $a\in R_{|a|}, b\in R_{|b|}$ where  $|a|,|b|\in \bbz_2$. We will denote by $\sak$ the category of commutative $\tbk$-superalgebras.  Furthermore, we keep the following notations and assumptions unless other stated.

\begin{itemize}

\item[(1)] Recall that the general linear supergroup $\GL(m|n)$ can be defined as an (affine) algebraic supergroup scheme over $\bbz$ (for example, see \cite[Chapter 5]{FG12} or \cite[\S3]{Shi}). In this sense, we write  $G_\bbz=\GL(m|n)_\bbz$. Set  $G_A=G(\bbz)_A$ for any (commutative) $\bbz$-algebra $A$ and $G=G_\tbk$ in the same spirit as in \cite[\S{I}.1.1]{Jan3}. This is to say, $G_A$ is a representable functor from the category of commutative $A$-superalgebras to the category of groups.
    Furthermore, its closed subgroups corresponding to the above subgroups appearing above are well defined as split algebraic $\bbz$-supergroups (see \cite[Chapter 4]{FG12}).
        Similarly, we have $T_A$, $B_A$, and $(G_\ev)_A$.
    In particular, $G=(G_\bbz)_\bk$. Denote by $\texttt{Dist}(G)$ the distribution algebra of $G$ (see  \cite{BKu} or \cite{Jan3}).



\item[(2)] {$T$: the standard maximal torus of $G=\GL(m|n)(=\GL(V)$ for $V=V_\bz\oplus V_\bo$ of superdimeision $(m|n)$) which is actually the usual maximal torus of $\GL(m)\times \GL(n)$ consisting of diagonal matrices associated with a fixed basis $v_1,\ldots,v_m$ of the even part $V_\bz$; $v_{m+1},\ldots,v_{m+n}$ of the odd part $V_\bo$.

\item[(3)]      The character group  $X(T)=\Hom(T,\Bbb{G}_{\text{m}})$ can be  identified with the free abelian group on generators $\sigma_i$, $i=1,\cdots, m+n$ where $\sigma_i$ picks out the $i$th entry of a diagonal matrix.  Corresponding to the parity of $v_i$, we have the parity $|\sigma_i|$ of $\sigma_i$ which is defined to be $1$ if $i\in\{1,\ldots,m\}$ and $-1$ if $i\in\{m+1,\ldots,m+n\}$. Set a symmetric bilinear form on $X(T)$ by appointing that
\begin{align}\label{eq: bilinearform first}
(\sigma_i, \sigma_j) = (-1)^{|\sigma_i|}\delta_{i,j}
\end{align}
where $\delta_{i,j}$ stands for the Kronecker function.
      }


\item[(4)] $B=B^-$: the standard Borel subgroup with $B(R)$ consisting of {lower} triangular matrices of $\GL(m|n)(R)$ for $R\in\sak$; $B^+$: Borel subgroup opposite to $B$.

\item[(5)] $\Phi$, $\Phi^\pm$, $\Phi_\bz$, $\Phi_\bo$: the root system of $G$, the positive/negative root system of $G$ associated with the standard Borel subgroup, the even root sets of $\Phi$, the odd root sets of $\Phi$ respectively.

\item[(6)] $\Pi$, $\Pi_\bz$, $\Pi_\bo$: the  fundamental simple root system, the even simple roots system, the  odd simple roots system.

\item[(7)] $\rho_\bz,\rho_\bo$: half the sum of positive even roots, half the sum of positive odd roots, respectively; $\rho:=\rho_\bz-\rho_\bo$.


\item[(8)]  $r_\alpha$, $W$: even reflection with
$r_{\alpha}(\mu)=\mu-2(\mu,\alpha)(\alpha,\alpha)^{-1}\alpha, \alpha\in \Pi_\bz$, the Weyl group generated by $\{r_\alpha\mid\alpha \in
\Phi_\bz\}$.

\item[(9)] $ X^{+}(T):=\{\lambda\in X(T)\mid (\lambda,\alpha)\geq 0 \text{ for all } \alpha\in \Phi^+_\bz\}$ the set of dominant weight in $X(T)$.

    \item[(10)] For simplicity of notations, set $H^i(G\slash B',\lambda):=H^i(G/B', \scrL_{G\slash B'}(\textbf{k}_\lambda))$ (see Convention \ref{con: add conv}). Associated with a commutative $\bbz$-algebra $A$,  the sheaf cohomology $H^i(G_A/B'_A, \scrL_{G_A\slash B'_A}(A_\lambda))$
 will be  written as $H^i(G_A\slash B'_A,\lambda)$ or $H^i_A(G\slash B',\lambda)$.

 Some other conventions for induced modules over purely-even groups are also introduced in
 \S\ref{S4.2}.

 \item[(11)] For a finite-dimensional $T$-module $M$, $M$ is decomposed into a direct sum of weight spaces $M=\sum_{\lambda\in X(T)}M_\lambda$.

     \item[(12)] For a superscheme $X$ over $\tbk$, we denote by $\tbk[X]$ the global section of the structural sheaf of $X$.

         \item[(13)] {{(Frobenius kernels) Fix a positive integer $r$ we define the Frobenius morphism $\sfF^r(R):G(R)\rightarrow G_\ev(R)$ raising each matrix entry to the $p^r$th power for $R\in\sak$. Note that for any $a\in R_\bo$, $a^{p^r}=0$. So the morphism makes sense. Let $G_r$ denote the kernel of $\sfF^r$ (the $r$th Frobenius kernel). Then $G_r$ becomes a normal subgroup scheme of $G$.}}
\end{itemize}

\subsection{}\label{2.2s}
 Keep $G= \GL(m|n)$ in the remainder of this section.  For $\lambda\in X(T)$, we have $\lambda=\sum_{1\leq i\leq m+n}\lambda_i\sigma_i$, where $~\sigma_i(t)=t_i, t=\begin{pmatrix}  t_1 &  0& \ldots & 0 \\ 0 & t_1 & \ldots & 0\\ \vdots & \vdots & \ldots & 0\\  0 & 0&\ldots&t_n \end{pmatrix} \in T(R)$.
 Then $\Phi=\{\sigma_i-\sigma_{j}\mid 1\leq i\neq j\leq m+n\}$ is the root system for $G$.

 Call $B^+$ the standard Borel subgroup. Associated with $B^+$ the (standard) positive root system   $\Phi^+=\{\sigma_i-\sigma_{j}\mid 1\leq i<j\leq m+n\}$ and the (standard) simple root system $\Pi=\{\sigma_i-\sigma_{i+1}\mid 1\leq i \leq m+n-1\}$.

In order to distinguish  even and odd roots, we change the notations by setting $\delta_i=\sigma_i, 1\leq i\leq m$ and $\epsilon_j=\sigma_{j+m}, 1\leq j\leq n$. Then $\Pi_\bo=\{\delta_m-\epsilon_1\}$ and $\Pi_\bz=\{\delta_i-\delta_{i+1},\epsilon_j-\epsilon_{j+1}\mid 1\leq i\leq m-1; 1\leq j\leq n-1\}$.
Then the bilinear form on $X(T)\otimes_{\mathbb Z}\mathbb Q$ coming from (\ref{eq: bilinearform first})
becomes $$(\delta_i, \delta_j)=\begin{cases}
      1 & \hbox{ if } 1\leq i=j\leq m,\\
      0  & \hbox{ if } i\neq j;
    \end{cases}\hbox{  }(\epsilon_i, \epsilon_{j})=\begin{cases}
      -1 & \hbox{ if } 1\leq i=j\leq n,\\
      0  & \hbox{ if } i\neq j.
    \end{cases}$$

\subsubsection{Odd reflections}\label{2.2.1} In the standard positive root system mentioned above, there are $mn$ positive odd roots: $\delta_s-\epsilon_t$ with $s$ ranging through $\{1,\ldots,m\}$ and $t$ ranging through $\{1,\ldots,n\}$.
We make them in an order by defining $\delta_s-\epsilon_t\prec \delta_k-\epsilon_l$ if and only if either $s>k$, or $s=k$ but $t<l$.
This order is identical to  the one defined in \cite[\S4]{BKu}.
By this order, the set of positive odd roots  can be parameterized by the positive integers  from $1$ to $mn$ as below,
 $\beta_1:=\delta_m-\epsilon_1, \beta_2:=\delta_{m}-\epsilon_2,\ldots,\beta_n:=\delta_{m}-\epsilon_n, \beta_{n+1}:=\delta_{m-1}-\epsilon_1,\beta_{n+2}:=\delta_{m-1}-\epsilon_2,
 \ldots,\beta_{2n}:=\delta_{m-1}-\epsilon_n,
 \beta_{2n+1}:=\delta_{m-2}-\epsilon_1,
\beta_{2n+2}:=\delta_{m-2}-\epsilon_2,\ldots,
\beta_{3n}:=\delta_{m-2}-\epsilon_n,\ldots,
\beta_{(m-1)n+1}:=\delta_{1}-\epsilon_1 ,\ldots,
\beta_{mn}:=\delta_{1}-\epsilon_n $.
 This is to say,  $\Phi_\bo=\{\beta_1,\beta_2,\ldots,\beta_{mn}\}$
 (there will be a concise presentation of $\beta_i$ in (\ref{eq: concept exp})). In contrast with the ordinary algebraic groups, there are different Borel subgroups which are not mutually conjugate. Recall that a Borel subgroup is dependent on its positive root system, equivalently, its simple root system.
Note that for any simple root system $\Pi'$ of $G$, there exists at least one odd root in $\Pi'$. Fix an odd root $\beta\in \Pi'$ which is certainly isotropic in the case of $G=\GL(m|n)$. Then, there is a new simple root system $\Pi''$ as below:
$$\Pi''=\{\alpha\in \Pi'\backslash \{\beta\}\mid (\alpha,\beta)=0\}\cup
 \{\alpha+\beta\mid \alpha\in \Pi', (\alpha, \beta)\ne 0\}\cup\{-\beta\}.$$
This operation changing $\Pi'$ into $\Pi''$ is just the odd refection associated with $\beta$, denoted by $\hat r_\beta$ (see \cite{LLS} or \cite[\S1.3]{CW}) .

If a simple root system arises from some other one via an odd reflection, then the corresponding Borel subgroups are not conjugate.

\subsubsection{Adjacent simple root systems via an ordered chain of odd reflections}\label{2.2.2}
For $\GL(1|1)$, the root system consists of  two odd roots $\pm(\delta_1-\epsilon_1)$. It is simple and clear. Let us consider the situation when $\GL(m|n)$ is not of type $(1|1)$. Recall that $\Pi$ contains  $\beta_1=\delta_m-\epsilon_{1}$.  Set $\Phi^+_{\beta_1}:=\{-\beta_1\}\cup \Phi^+\setminus\{\beta_1\}$.
Then $\Phi^+_{\beta_1}$
is a new positive root system (see for example \cite[Lemma 1.30]{CW}),
whose  fundamental root system is by computation
 $$\Pi_{\beta_1}=\{\delta_1-\delta_{2},\ldots,\delta_{m-2}-\delta_{m-1},\delta_{m-1}-\epsilon_1, \epsilon_1-\delta_m,\delta_m-\epsilon_2,\epsilon_{2}-\epsilon_{3},\ldots,\epsilon_{n-1}-\epsilon_{n}\}$$
if $n\geq 2$, which contains $\beta_2=\delta_m-\epsilon_2$,
or
$$\Pi_{\beta_1}=\{\delta_1-\delta_{2},\ldots,\delta_{m-2}-\delta_{m-1},
\delta_{m-1}-\epsilon_1, \epsilon_1-\delta_m\}$$
if $n=1$ (consequently $m\geq 2$), which contains $\beta_2=\delta_{m-1}-\epsilon_1$. Repeatedly by operation $\hat r_{\beta_2}$, the simple root system $\Pi_{\beta_1}$ is changed into
$$\Pi_{\beta_2}=\{\delta_1-\delta_{2},\ldots,\delta_{m-2}-\delta_{m-1},
\delta_{m-1}-\epsilon_1, \epsilon_1-\epsilon_2,\epsilon_2-\delta_m,\delta_m-\epsilon_{3},
\epsilon_{3}-\epsilon_{4},\ldots,\epsilon_{n-1}-\epsilon_{n}\}$$
 for the case $n\geq 2$, which contains $\beta_3=\delta_m-\epsilon_3$ provided that $n\geq 3$,
or
$$\Pi_{\beta_2}=\{\delta_1-\delta_{2},\ldots,\delta_{m-3}-\delta_{m-2},
\delta_{m-2}-\epsilon_1, \epsilon_1-\delta_{m-1},\delta_{m-1}-\delta_m\}$$
 for the case $n=1$ while $m\geq 2$, which contains $\beta_3=\delta_{m-2}-\epsilon_1$ provided $m\geq 3$.

The above process can be repeated. We have the following basic observation.

\begin{lemma}\label{lem: adjacent borels}
 There are $mn$ simple root systems $\Pi_{\beta_i}$, $i=1,\ldots,mn$ satisfying that for each $i$,  $\Pi_{\beta_{i+1}}$ is arising from $\Pi_{\beta_{i}}$ by odd reflection $\hat r_{\beta_{i+1}}$.
\end{lemma}
\begin{proof} For the type $(1|1)$, there is nothing to say. Suppose $G=\GL(m|n)$ is not of type $(1|1)$. If $n=1$ (consequently, $m\geq 2$), then we can list  $\beta_i=\delta_{i'}-\epsilon_1$ for $i=1,\ldots,m$ and $i':=m-i+1$. And $\Pi_{\beta_i}=\{\delta_1-\delta_2,\ldots,\delta_{i'-2}-\delta_{i'-1},-\beta_i,\beta_{i+1},
 \delta_{i'}-\delta_{i'+1},\ldots, \delta_{m-1}-\delta_m\}$ for $i=1,\ldots,m-1$, and $\Pi_{\beta_m}=\{-\beta_m, \delta_1-\delta_2,\ldots, \delta_{m-1}-\delta_m\}$. In this case,  the statement obviously holds.

 In the following, suppose $n\geq 2$. We only need to show that $\Pi_{\beta_{i}}$ contains $\beta_{i+1}$ and $-\beta_{i}$ (when $i=mn$, $\beta_{i+1}$ is redundant).
Note that for $i\in\{1,\ldots,mn\}$, we can write $i=kn$ with $1\leq k\leq m$, or $i=kn+l$ with $0\leq k<m$ while $1\leq l<n$.  Set $k':=m-(k-1)$ for $1\leq k<m$. Then we can rewrite $\beta_i$ as below
\begin{align}\label{eq: concept exp}
\beta_i=\begin{cases} \delta_{k'}-\epsilon_n, \text{ if } i=kn;\cr
\delta_{k'-1}-\epsilon_l, \text{ if } i=kn+l.
\end{cases}
\end{align}

   We will inductively present $\Pi_{\beta_i}$. Hence the lemma follows from such a precise presentation.
 When $i=1$ and $i=2$, it is clear by the arguments in the paragraph before the lemma. For the general case, we proceed by dividing into cases of variant forms of $i$.

When $i=kn$ with $1\leq k\leq m$, then
\begin{align*}
\Pi_{\beta_i}=\{\delta_1-\delta_2,\ldots, \delta_{k'-2}-\delta_{k'-1}, \cr
&\delta_{k'-1}-\epsilon_1(=\beta_{i+1}), \cr \epsilon_1-\epsilon_2,\ldots,\epsilon_{n-1}-\epsilon_n; \cr &\epsilon_n-\delta_{k'}(=-\beta_i),\cr \delta_{k'}-\delta_{k'+1},\ldots,\delta_{m-1}-\delta_m
\}& \text{ if }k<m,
\end{align*}
and
\begin{align*}
\Pi_{\beta_{mn}}=\{\epsilon_1-\epsilon_2,\ldots,
\epsilon_{n-1}-\epsilon_n,\cr
\epsilon_n-\delta_1 (=-\beta_{mn}), \cr \delta_1-\delta_2,\ldots,\delta_{m-1}-\delta_m\}.
\end{align*}

Now we proceed to work with $i=kn+l$ with $0\leq k<m$ and $1\leq l<n$. We need some additional convention. When $k=0$ we additionally appoint $k'-2=(k+1)'-1=m-1$, $k'-1=(k+1)'=m$ (note that $k'$ for $k=0$ does not make sense. Naturally, any term indicated by $k'$ is  redundant).
Then
\begin{align*}
\Pi_{\beta_i}=\{\delta_1-\delta_2,\ldots,\delta_{k'-2}-\epsilon_1, \epsilon_1-\epsilon_2,\ldots,\epsilon_{l-1}-\epsilon_l, \cr &\epsilon_l-\delta_{k'-1} (=-\beta_{i}),\cr
 &\delta_{k'-1}-\epsilon_{l+1} (=\beta_{i+1}),\cr
\epsilon_{l+1}-\epsilon_{l+2},\ldots,\epsilon_{n-1}-\epsilon_n,
\epsilon_n-\delta_{k'},\delta_{k'}-\delta_{k'+1},\ldots,&\delta_{m-1}-\delta_m\}
\text{ if }l+1<n,
\end{align*}
and
\begin{align*}
\Pi_{\beta_i}=\{\delta_1-\delta_2,\ldots,\delta_{k'-2}-\epsilon_1, \epsilon_1-\epsilon_2,\ldots,\epsilon_{l-1}-\epsilon_l\cr &\epsilon_l-\delta_{k'-1} (=-\beta_{i}),\cr
& \delta_{k'-1}-\epsilon_{n} (=\beta_{i+1}),\cr
\epsilon_n-\delta_{k'},\delta_{k'}-\delta_{k'+1},\ldots,&\delta_{m-1}-\delta_m\}
\text{ if }l+1=n.
\end{align*}

By the above arguments, it has been verified that $\Pi_{\beta_{i}}$ contains $-\beta_i$ and $\beta_{i+1}$. The proof is completed.
\end{proof}

The positive root system corresponding to $\Pi_{\beta_i}$ will be denoted by $\Phi^+_{\beta_i}$ which contains the subset $(\Phi^+_{\beta_i})_\bo$ of odd roots. The corresponding Borel subgroup is denoted by ${B^{(i)}}^+$. The opposite one will be denoted by $B^{(i)}$ ($={B^{(i)}}^-$).

\subsubsection{Super Weyl groups and their longest elements}\label{2.2.3} {The super Weyl group $\widehat W$  is introduced in \cite[\S3.2]{PS}, which is by definition a subgroup of the transform group of the set of Borel subgroups containing $T$  generated by all odd reflections defined previously along with real (ordinary) reflections (called even reflections later).}  Then  $\widehat W$ contains the ordinary Weyl group $W$ as a subgroup.

Set $w_{\bar 1}=\hat r_{mn}\hat r_{mn-1}\cdot\cdot\cdot\hat r_1$  where $\hat r_{i}=\hat r_{\beta_i}$.  Denote by $w_{\bar 0}$ the longest element of $W$. By  Theorem \cite[Theorem 3.10]{PS}, the following result is readily deduced.
%

 \begin{lemma}\label{lem: longest Weyl ele} The longest element $\widehat w_\ell$ of $\widehat W$ can be written as $w_{\bar 0}w_{\bar 1}$, which changes the standard Borel group $B^+$ associated with $\Phi^+$ into $B^-$ associated with $-\Phi^+$.
\end{lemma}

    {
\subsubsection{Flag superschemes}\label{sec: flag pre} A general notion of flag superschems can be seen in \cite[\S2.4-2.5]{BHP}. Now we introduce the flag superschemes we will use. Keep the notations in \S\ref{sec: conv}(13).
For each $R\in\sak$, let $\tilde V(R):=\Hom_\bk(V,R)$. For a morphism $\theta: R\rightarrow R'$ in $\sak$, define $\tilde V(\theta):f\mapsto \theta\circ f$. Then $\tilde V$ becomes an affine superscheme, and  $\tilde{V}(R)$ becomes a free left $R$-module of rank $(m|n)$ with action defined by $(a.f)(v)=a(f(v))$ for $a\in R$. The natural representation of $G=\GL(V)$ on $V$ gives rise to the right action of $G$ on $\tilde{V}$.
 A flag in $\tilde{V}(R)$ is defined to be a chain $\mathfrak{f}:=(\fff_1\subset \fff_2\subset \cdots\subset \fff_{m+n})$ of direct summands of the free $R$-supermodule $\tilde{V}(R)$ with $\fff_i$ has rank $(m_i|n_i)$ with $m_i\leq m_{i+1}$ and $n_i\leq n_{i+1}$ and $m_i+n_i+1=m_{i+1}+n_{i+1}$. Let $\tilde X$ denote the functor mapping a superalgebra $R\in \sak$ to the set $\tilde{X}(R)$ of  all flags in $\tilde{V}(R)$. For a  morphism $\theta:R\rightarrow R'$, $\tilde{X}(\theta):\tilde{X}(R)\rightarrow \tilde{X}(R')$ is the map induced by composing with $\theta$. It is readily seen that the right action of $G$ on $\tilde V$ induces a right action on $\tilde X$.  Correspondingly, we get a corresponding orbit map $\tilde\pi:G\rightarrow \tilde X$ which will be used in \S\ref{sec: third}.

}

\section{Induced modules  and Serre Duality}

In this section, we first introduce the general notations and conventions on algebraic supergroups and  induced modules, which is referred to \cite[\S2]{BK03}) and \cite[\S2]{B06}).
 Then we introduce Serre duality for  higher cohomology arising from induced functor.  Suppose $G$ is a given algebraic supergroup, and $H$ its closed subsupergroup.

 \subsection{}Let $\mu$ be the multiplication in $G$. Denote by $\mu^{\#}$ the comorphism from $\mathscr{O}_G\rightarrow \mu_*(\mathscr{O}_G\otimes\scrO_G)$.
 We denote by $G\hmod$ and $H\hmod$ the supermodule categories of $G$ and $H$ respectively.  As an analogue of algebraic group case, there are induction and restriction functors:
$$\texttt{ind}_H^G:H\hmod\rightarrow G\hmod; \;\; \texttt{res}_H^G:G\hmod\rightarrow H\hmod $$
 (see \cite[\S{I.3}]{Jan3}, \cite[\S2]{B06}, \cite[\S6]{BK03}). Then $\texttt{res}$ is exact. And $\texttt{ind}_H^G$ is left exact, and  right adjoint to $\texttt{res}^G_H$. {Denote by ${R}^i\texttt{ind}^G_H$ the $i$th right derived functor of $\texttt{ind}^G_H$.}
  Suppose $M$ (resp. $N$) is a  (left) $H$-supermodule (resp. $G$-supermodule). This is equivalent to say, {$M$ is a (right) supercomodule} over $\bk[H]$ with structure map $\eta: M\rightarrow M\otimes \bk[H]$. Then there is a natural
isomorphism (a generalized tensor identity):
\begin{align}\label{eq: gen ten id}
R^i\texttt{ind}^G_H(\texttt{res}^G_H M\otimes N)\cong  M \otimes R^i\texttt{ind}^G_H N
\end{align}

\subsection{}\label{sec: Brundan's setting} Suppose there is a quotient $X$ of $G$ by $H$ with the defining morphism: $\pi: G\rightarrow X$ which satisfies (Q1)-(Q6) in \cite[\S2]{B06}. This is to say, $G/H$ is locally decomposable and $G_\ev\slash H_\ev$ is projective; etc.  Under these assumptions, we further have  the underlying purely-even scheme $X_\ev$ which  is the scheme over $\bk$ equal to $X$ itself as a topological space, with structure sheaf $\scrO_X\slash \mathscr{I}_X$. Here $\mathscr{I}_X$ is the quasi-coherent sheaf of superideas on $X$ which is defined via the presheaf sending any open subset $U$ to $\scrO_X(U)\scrO_X(U)_\bo$.
 Let $\rho: X\times G\rightarrow X$ be the right action of $G$ on $X$ induced by multiplication in $G$.  As an ordinary case, one can define a $G$-equivariant $\mathscr{O}_X$-supermodule (see \cite{B06}).
 By definition, a quasi-coherent $\scrO_XG$-supermodule means a quasi-coherent $\scrO_X$-supermodule $\mathcal{M}$ equipped with an even $\scrO_X$-supermodule map $\iota:\mathcal{M}\rightarrow \rho_*(\mathcal{M}\otimes \scrO_G)$ satisfying the compatibility axioms with respect to the $G$-action (see \cite[\S2]{B06}).
 Denote by $\mathscr{O}_XG\hmod_{\text{qcoh}}$ the category of  $G$-equivariant quasi-coherent $\mathscr{O}_X$-supermodules.
For example, $\pi_*\mathscr{O}_G$ is a quasi-coherent $\mathscr{O}_XG$-module with structure map $\pi_*\mu^{\#}:\pi_*\mathscr{O}_G\rightarrow \rho_*(\pi_*\mathscr{O}_G\otimes \mathscr{O}_G)$. Also,
$\bk[H]_{\text{trivial}}\otimes \pi_*\mathscr{O}_G$ is a quasi-coherent $\mathscr{O}_XG$-module with structure map $\id_{\bk[H]}\otimes\pi_*\mu^{\#}$.

 Furthermore, as the ordinary case, there is a functor  $\mathscr{L}$ from $H\hmod$ to the category  $\mathscr{O}_XG\hmod_{\text{qcoh}}$ of quasi-coherent $\mathscr{O}_XG$-supermodules as blow. Suppose $M\in H\hmod$, then $M$ is a (right) cosupermodule over $\bk[H]$ with structure map $\eta: M\rightarrow M\otimes \bk[H]$. Define the quasi-coherent $\mathscr{O}_XG$-supermodule $\mathscr{L}(M)$ to be the kernel of the map $\eta\otimes \id_{\pi_*\mathscr{O}_G}-\id_M\otimes \delta$ where $\delta=\pi_*\bar\mu^{\#}$ is a natural $\mathscr{O}_XG$-supermodule map from $\pi_*\mathscr{O}_G$ to $\bk[H]_{\text{trivial}}\otimes \pi_*\mathscr{O}_G$, and
 $\bar\mu$ is the restriction of $\mu$ to $H\times G$.
Actually, $\mathscr{L}(M)$ is just the associated sheaf of $X$ with respect to $M\in H\hmod$ (see \cite[\S{I.5.8}]{Jan3} and \cite[\S2]{B06}).

\subsection{}   By \cite[Corollary 2.4]{B06}, we have
 \begin{align}\label{eq: ind and ass}
 {R}^i \texttt{ind}^G_H(-)\cong H^i(X,-)\circ \mathscr{L}.
 \end{align}

 Write ${\textsf{i}}:X_\ev\rightarrow X$ for the canonical closed immersion which is  $G_\ev$-equivariant.  There is a natural restriction functor  $\texttt{res}^G_{G_\ev}: \scrO_{X}G\hmod_{\text{qcoh}}\rightarrow  \scrO_{X}G_\ev\hmod_{\text{qcoh}}$. And the functor $\scrL_\ev$ can be defined as a restriction of $\scrL$ to $H_\ev\hmod$. Then $\scrL_\ev$ is a  functor from $H_\ev\hmod$ to the category  $\scrO_{X_\ev}G_\ev\hmod_{\text{qcoh}}$.

 \begin{lemma}\label{lem: even res}
  (\cite[\S2(6) and Lemma 2.5]{B06}) For $M\in H\hmod$, and $\mathcal{M}\in\scrO_XG\hmod$, the following statements hold.
  \begin{itemize}
  \item[(1)]  ${\textsf{i}}^{\,*}(\texttt{res}^G_{G_\ev}\scrL(M))\cong \scrL_\ev(\texttt{res}^H_{H_\ev}M)$.
 \item[(2)] $H^i(X, \texttt{res}^G_{G_\ev}\mathcal{M})\cong \texttt{res}^G_{G_\ev} H^i(X,\mathcal{M})$.
 \end{itemize}
 \end{lemma}

\subsection{}\label{sec: third} In the following we suppose $G=GL(m|n)$ . 
Consider $X=G\slash B$ for any Borel subgroup $B$ of $G$.
In this case, the purely-even scheme $X_\ev$ is just a quotient of $G_\ev$ by $B_\ev$, a flag variety.  Basically, $G_\ev\slash B_\ev$ is projective and G/B is locally decomposable. Consequently, Brundan's assumptions (Q1-Q6) satisfy (cf.  \cite[\S4.4]{MT}), we have a quotient $\pi: G\rightarrow X$.     {A flag superscheme and projective superscheme associated with $X$ can be considered (see \cite{BHP} for general definitions and properties). Precisely, we can proceed with arguments on flag superschemes of $X$ which is actually  $\tilde X$ in \S\ref{sec: flag pre}, by exploiting the arguments in \cite[\S3]{B06}.}

 We remind again, for simplicity, we will  omit the adjunctword ``super" for   super-modules, super-subgroups, etc.

    {We could not have found any satisfactory reference for the Serre duality for the flag superscheme $X$ in characteristic $p$ until we found very recently-published paper \cite[\S3.5]{BHP} which deals with general superschemes.  On the other hand, as was done by Brundan in \cite[\S5]{B06}, we  can state  the super version of Serre duality theorem with the sheaf cohomology on $X$, mimicking the arguments for the version on complex supermanifold of classical Serre duality in  \cite{OP} (or the ones on complex superschemes in  \cite[\S2]{VMP}).

%

We first recall that the Berezinian $\texttt{Ber}_X(S)$ of a free $\co_X$-module $S$ of rank $(l|q)$   is by definition a free $\co_X$-module  of rank $(1|0)$ (if $q$ is even) or $(0|1)$ (if $q$ is odd), which is functorial with respect to even isomorphisms of $\co_X$-modules and coincides with $\bigwedge^lS$ for $q=0$ (see \cite[\S3.4]{Man}). One has a simple generalization of the above $\texttt{Ber}$ to the case of a locally free $\co_X$-module. Set $\texttt{Ber}(X)=(\texttt{Ber}_X \Omega^1X)^*$, where $\Omega^1X$ stands for the differential form of degree $1$.

Next, denote by $E^*$ the dual space of $\scrI_X\slash \scrI_X^2$, and by $\Omega^{l(w_\bz)}(X_\ev)$ the differential form of degree $l(w_\bz)=\dim X_\ev$. By the same arguments as in \cite[Lemma 1]{OP},  we have $$\texttt{Ber}(X)\cong
\bwedge(E^*)\otimes_{\scrO_{X_\ev}}\Omega^{l(w_\bz)}(X_\ev).$$
Furthermore, we have
\begin{lemma}\label{lem: Ber}
$$\texttt{Ber}(X)\cong 
\scrL_\ev(\bwedge(\ggg\slash \bbb^-)_\bo\otimes \bk_{-2\rho_\bz}).$$
\end{lemma}
\begin{proof}
By \cite[Lemma 2.6]{B06}, $E^*\cong \scrL_\ev((\texttt{Lie} G\slash \texttt{Lie} B)_\bo)$. Hence $$\bwedge(E^*)\cong \scrL_\ev(\bwedge(\ggg\slash \bbb^-)_\bo).$$
By a known result in the case of reductive algebraic groups, we have $\Omega^{l(w_\bz)}(X_\ev)\cong \scrL_\ev(\bk_{-2\rho_\bz})$ (cf.  \cite[\S{II.4.2(6)}]{Jan3}). The lemma follows.
\end{proof}
 }

    {Note that the canonical projective morphism $f: X\rightarrow \text{Spec}(\tbk)$ is a proper smooth  morphism of superschemes of finite type, which has relative dimension $(\ell(w_\bz)\mid \dim \bbb_\bo^-)$. Then $\texttt{Ber}(X)$ can be regarded as $\Omega^1_{X\slash \text{Spec}(\tbk)^{\ev}}$ the sheaf of relative even differentials. By the same arguments as in \cite[\S2.2]{VMP}, we have the following version of Proposition 2.3 of \cite{VMP} on our case.}

\begin{theorem}\label{Serre duality}
 Keep the notations as above. Then $\texttt{Ber}(X)$ is a dualizing sheaf of $X$, and there is a natural isomorphism
$$ (R^i\texttt{ind}^G_BM )^*\cong R^{l(w_\bz)-i}\texttt{ind}^G_B (M^*\otimes \texttt{Ber}(X)). $$
\end{theorem}

    {
\begin{remark} The reader can refer to \cite[Propositions 3.22 and 3.24]{BHP} for general theory of super version of Serre duality,  comparing it with  the ordinary case \cite[III.7.7]{Ha} for projective schemes and \cite[II.4.2]{Jan3} for flag varieties in prime characteristic.
\end{remark}
}

\subsection{Super analogue of Mackey imprimitivity theorem} Keep the notations and assumptions in \S\ref{sec: Brundan's setting}. Additionally, suppose $L$ is an affine supergroup scheme which is a closed subsupergroup of $G$, there is a morphism of supergroup schemes $L\rightarrow X=G\slash H$ giving rise to  epimorphism of structural sheaves. Then one has a super analogue of Mackey imprimitivity theorem \cite[Theorem 4.1]{CPS} for algebraic groups.
\begin{theorem}(\cite[Theorem 10.1]{Z1})\label{thm: zubakov}
 Keep the above notations and assumptions. Then for any $H$-supermodule $M$, there is an isomorphism of $L$-supermodules
$$ \texttt{Res}^G_L(R^i\texttt{ind}^G_H(M))\cong R^i\texttt{ind}^L_{L\cap H}(\texttt{Res}^H_{L\cap H}(M)).$$
\end{theorem}

\subsection{Induced modules and their socles}\label{S4.2}
Keep the notations and assumptions as above. For a $B$-module $M$  we  by convention write $H^i(M)$ and $H^i_\ev(M)$ for $H^i(G\slash B, \scrL(M))$ and for $H^i(G_\ev\slash B_\ev, \scrL_\ev(M))$ respectively.

 Let $\textbf{k}_{\lambda}$ be the one-dimensional $B$-module of weight $\lambda$ for $\lambda\in X(T)$.
 As a usual way we write $H^i(\lambda)=H^i(\textbf{k}_\lambda))$ and $H^i_\ev(\lambda)=H^i_\ev(\textbf{k}_\lambda)$. According to \eqref{eq: ind and ass}, we have $H^i(\lambda)\cong R^i\ind^G_B(\bk_\lambda)$.
Especially, $H^0(\lambda)\cong\texttt{ind}^G_B\bk_\lambda$. In the subsequent, we will identify $\texttt{ind}^G_B\bk_\lambda$ with $H^0(\lambda)$.
Set $L(\lambda):=\texttt{Soc}_G(\texttt{ind}^G_B\bk_\lambda)$.

Furthermore, we will use the following convention.
\begin{convention}\label{con: add conv}
 Often simply write $H^i(G/B, \lambda)$ for $H^i(G/B, \scrL(\textbf{k}_\lambda))$, even write $H^i(G/B', \lambda)$ for $H^i(G/B', \scrL_{G\slash B'}(\textbf{k}_\lambda))$ where $B'$ is a given Borel subgroup, and $\scrL_{G\slash B'}(\textbf{k}_\lambda)$ is an associated sheaf on $G\slash B'$.
\end{convention}

 The following results are important for the subsequent arguments .
\begin{theorem}\label{thm: Shi thm}
 \begin{itemize}
 \item[(1)] Up to an isomorphism of $G_\ev$-modules, $H^i_\ev(\lambda)$ can be regarded a submodule of $H^i(\lambda)$ for any $i$ and $\lambda\in X(T)$. This statement holds for any Borel subgroup containing $T$.
 \item[(2)] Let ${B'}^+$  be a Borel subgroup corresponding to a positive root system ${\Phi'}^+$, $B'={B'}^{-}$ its opposite Borel. Set ${\bbb'}^-=\texttt{Lie}({B'})$. Then for any $B'$-module $M$, there is a $T$-module isomorphism
\begin{align}\label{eq: gen coho decom}
R^{i}\texttt{ind}^G_{B'}(M)\cong R^{i}\texttt{ind}^{G_\ev}_{B'_\ev}(M) \otimes \bwedge(\ggg\slash {\bbb'}^-)^*_\bo.
\end{align}
\end{itemize}
\end{theorem}

\begin{proof} We first make an account of the first part of (1). By the same arguments for any Borel subgroup containing $T$, the second part (1) can be verified. Note that $\texttt{ind}^{G_\ev}_{B_\ev}$ is left exact. It is sufficient to show there is an injective homomorphism of  $G_\ev$-modules from $H^0_\ev(\lambda)$ into $H^0(\lambda)$. Recall
\begin{align*}
H^0(\lambda)=\{f\in \tbk[G]\mid& f(gb)=\lambda(b)^{-1}f(g) \;
\forall g\in G(R), b\in B(R)\cr &\text{ for any } R\in\sak\}.
\end{align*}
    {Recall that the first Frobenius kernel $G_1$ is a normal subgroup scheme of $G$ (see \S\ref{sec: conv}(13)).
And $G_\ev$ is a closed subgroup scheme of $G$, and there exist set-theoretic factorization $G(R)=G_\ev(R)G_1(R)$ for any $R\in \sak$ where $G_1(R)$ is a normal subgroup of $G(R)$ generated by one-parameter subgroups arising from odd root vectors}
in somewhat ``odd"-way (see \cite[Theorem 5.15]{FG12}).
So there is a natural way to define  a homomorphism of $G_\ev$-modules from $H^0_\ev(\lambda)$ to $H^0(\lambda)$. Note that as a topological space, open subsets of $G$ are by definition just the ones of $G_\ev$. So it follows that this homomorphism is injective.

    {
As to (2), it's actually simple generalization of a result on standard Borels (cf. \cite[Theorem 5.5]{Shi}, \cite[Proposition 5.2]{Z2}).  We simply make it into an account by exploiting the arguments in the proof of  \cite[Theorems 5.4 and  5.5]{Shi} to our case with  Borel $B'$.  Let $P'$ be a parabolic subgroup of $G$ such that $P'_\ev=G_\ev$ and $\Lie(P')_\bo={{\bbb_\bo'}^-}$. Then by a direct computation involving  Hopf algebras Proposition 5.2 of \cite{Shi} is valid to the case $P'$, which means, $\tbk[G]$ is isomorphic to
$$\tbk(P') \otimes  \bwedge(\ggg\slash {\bbb'}^-)^*_\bo$$
as right $P'$-left $T$-modules. Then Theorem 5.4 of \cite{Shi} and consequently Theorem 5.5 therein still hold with  $B$ being taken place by $B'$.}
\end{proof}

 As in the case of reductive algebraic groups, one has that $H^0(\lambda)$ is finite-dimensional (cf. \cite[Corollary 5.2]{Z2} or \cite[Proposition 4.17]{Shi}). As mentioned above,  $L(\lambda)=\texttt{soc}(H^0(\lambda))$. Regarded as a $B^+$-module, $\texttt{soc}_{B^+}(L(\lambda))$ is precisely the $\lambda$-weight space $H^0(\lambda)_\lambda$ coinciding with  $\tbk_\lambda$ (cf. \cite[Proposition 4.11 ]{Shi}). Consequently, $L(\lambda)$ is simple $G$-module. Furthermore,
 irreducible modules $\{L(\lambda) \mid \lambda\in X^+(T)\}$
form a complete representative set of isomorphism classes of irreducible  modules of $G$ (see \cite[Theorem 4.5]{BKu} or \cite[Theorems 4.12 and 5.5, Example 5.9]{Shi}).
  One can also describe irreducible $G$-modules by considering a kind of $T$-compatible $\texttt{Dist}(G)$-module category (see \cite{BKu}, \cite{CSW}, \cite{SW}, etc.).
Furthermore,  all these modules $\{L(\lambda) \mid \lambda\in X^+(T)\}$ are of type $\texttt{M}$ (i.e. $\mbox{End}_{G}({L}(\lambda))$ is 1-dimensional).

We sum up with the following theorem.

\begin{theorem} \label{lem8.3} Suppose that $\lambda\in X(T)$. The following statements hold.
\begin{itemize}
\item[(1)] $H^0(\lambda)$ is finite-dimensional.
\item[(2)]  $H^0(\lambda)$ is non-zero if and only if $\lambda \in X^+(T)$.
\item[(3)] The socle of $B^+$-module $H^0(\lambda)$ is precisely its $\lambda$-weight space $H^0(\lambda)_\lambda$ coinciding with  $\bk_\lambda$. And $L(\lambda)$ is simple $G$-module.
\item[(4)] The modules $\{L(\lambda) \mid \lambda\in X^+(T)\}$ form a complete set of isomorphism classes of irreducible  $G$-modules.
\item[(5)] All modules $\{L(\lambda) \mid \lambda\in X^+(T)\}$ are of type $\texttt{M}$.
\end{itemize}
\end{theorem}

\section{Induced modules: Case $\GL_A(1|1)$}\label{s3}
In this section, we assume $G=\GL(1|1)$. In this case,   there are two Borel subgroups  containing the standard torus. Both of them are  not mutually conjugate under $G$-conjugation. In this section, we  introduce  the construction of induced modules associated with different Borel subgroups. The material here mainly comes from \cite[\S 7.3]{MZ} and \cite{Z1}, which will be important to the subsequent sections.

\subsection{} We are given a commutative $\bk$-superalgebra $R$. Recall that $G(R)$ consists of invertible  matrix $\begin{pmatrix}  a &  m \\  n &  b \end{pmatrix}$ with $a,b\in R_\bz$ and $m,n\in R_\bo$. Then
$G(R)$ is generated by $\begin{pmatrix}  a &  0 \\  0 &  b \end{pmatrix}, \begin{pmatrix}  a &  m \\  0 &  b \end{pmatrix}$ and $\begin{pmatrix}  a &  0 \\  n &  b \end{pmatrix}$. In this case,  $G$ has only two Borel subgroups $B_{\beta}^+$ and $B_{-\beta}^+$, corresponding to $\Phi_{\beta}^+=\{\beta:=\delta_1-\epsilon_1\}$ and $\Phi_{-\beta}^+=\{-\beta=\epsilon_1-\delta_1\}$, respectively.
 Denote $B_{-\beta}=B_{\beta}^-$, $B_\beta=B_{-\beta}^-$.
 Let $\textbf{k}_\lambda$ be the one-dimensional $B_{\pm\beta}$-module via the torus action by weight $\lambda$.  We will simply write a weight $\lambda=i\delta_1+j\epsilon_1$ as $(i|j)$ in the following.

Recall $H^i(\lambda)\cong R^i\ind^G_B(\bk_\lambda)$. For $\GL(1|1)$, in order to avoid confusion,
we denote $H_{\pm\beta}^i(\lambda )=R^i\ind^G_{B_{\pm\beta}}(\bk_\lambda)$.
Denote by $L_{B_{\pm\beta}}(\lambda)$ the socle of $H_{\pm\beta}^0(\lambda )$.



\subsection{}\label{3.2}
We define a function $\bc_{ij}^k$ on  $2\times 2$-matrices via  $\bc_{ij}^k\begin{pmatrix}  a_{11} &  a_{12} \\  a_{21} &  a_{22} \end{pmatrix}=a^k_{ij}$. Then we have
$$\bc_{ij}^k[\begin{pmatrix}  a_{11} &  a_{12} \\  a_{21} &  a_{22} \end{pmatrix}\cdot\begin{pmatrix}  b_{11} &  b_{12} \\  b_{21} &  b_{22} \end{pmatrix}
]
=(a_{i1} b_{1j}+ a_{i2} b_{2j} )^k.$$
 Recall that the coordinate ring $\tbk[G]$ is the localization of  the polynomial ring $\tbk[\bc_{ij}\mid i,j=1,2]$ at $\texttt{det}=\bc_{11}\bc_{22}$. It is a Hopf superalgebra with comultiplication $\Delta$ via  $\Delta(\bc_{st})=\sum_{k=1,2}\bc_{sk}\otimes \bc_{kt}$, $s,t=1,2$.

 Set $\textbf{k}[G]_\lambda$ to be the subspace of $\textbf{k}[G]$ of weight $\lambda$, which is   spanned by all monomials of weight $\lambda$.
For example, take $\digamma=\bc_{11}^{a}\bc_{12}^{b}\bc_{21}^{c}\bc^d_{22}, a,b,c,d\in\mathbb Z$ and $0\leq b, c\leq 1$. The weight of $\digamma$  is $\lambda=(a+b|c+d)$.

\subsection{}\label{3.3} For $\lambda=(i|j)$, we have $(\lambda,\beta)=i+j$, denoted by $|\lambda|$.    For  $\textbf{k}[G]=\textbf{k}[\bc_{ij}|1\leq i, j\leq 2]_{\bc_{11}\bc_{22}}$, set
 $$A_\lambda=\bc_{11}^{i}\bc_{22}^{j}, B_\lambda=\bc_{11}^{i-1}\bc_{12}\bc_{22}^j, C_\lambda=\bc_{11}^{i}\bc_{21}\bc_{22}^{j-1}, D_\lambda=\bc_{11}^{i-1}\bc_{12}\bc_{21}\bc_{22}^{j-1}.
 $$

By computation,
 \begin{align*}
 \Delta(A_\lambda)&=\Delta(\bc_{11})^i\cdot\Delta(\bc_{22})^j\cr
 &=(\bc_{11}^i\otimes \bc_{11}^{i}+i\bc_{11}^{i-1}\bc_{12}\otimes \bc_{11}^{i-1}\bc_{21})(\bc_{22}^j\otimes \bc_{22}^{j}+j\bc_{21}\bc_{22}^{j-1}\otimes \bc_{12}\bc_{22}^{i-1})\cr
 &=\bc_{11}^i\bc_{22}^j\otimes \bc_{11}^i\bc_{22}^j+ j\bc_{11}^i\bc_{21}\bc_{22}^{j-1}\otimes \bc_{11}^i\bc_{12}\bc_{22}^{j-1}+\bc_{11}^{i-1}\bc_{12}\bc_{22}^j\otimes \bc_{11}^{i-1}\bc_{21}\bc_{22}^j\cr
 &\;\;\;+ij\bc_{11}^{i-1}\bc_{12}\bc_{21}\bc_{22}^{j-1}\otimes \bc_{11}^{i-1}\bc_{12}\bc_{21}\bc_{22}^{j-1}.
 \end{align*}

Thus $\Delta(A_\lambda)=A_\lambda\otimes A_\lambda+iB_\lambda\otimes C_{\lambda-\beta}+jC_\lambda\otimes B_{\lambda+\beta}+ijD_\lambda\otimes D_{\lambda}.$
Similarly, we have
  \begin{align*}&\Delta(B_\lambda)=B_\lambda\otimes Y_{\lambda-\beta}+X_\lambda\otimes B_{\lambda};\cr
  &\Delta(C_\lambda)=C_\lambda\otimes X_{\lambda+\beta}+Y_\lambda\otimes C_{\lambda};\cr
  &\Delta(D_\lambda)=D_\lambda\otimes A_\lambda-C_\lambda\otimes B_{\lambda+\beta}+B_\lambda\otimes C_{\lambda-\beta}+(A_\lambda+(j-i)D_\lambda)\otimes D_\lambda
  \end{align*}
  where $X_\lambda=A_\lambda+jD_\lambda, Y_\lambda=A_\lambda-iD_\lambda$.

\subsection{}\label{3.4} There is a natural action of $G(R)$ for $R\in\sak$ on $\bk[G]\otimes_\bk R$ in the same spirit as in  the ordinary algebraic group case.
 By the computations in \S \ref{3.2} and \S \ref{3.3},
  $$\begin{pmatrix}  a & 0\\  0 &  b \end{pmatrix} C_\lambda= X_{\lambda+\beta}\begin{pmatrix}  a & 0\\  0 &  b \end{pmatrix}C_\lambda+ C_\lambda\begin{pmatrix}  a & 0\\  0 &  b \end{pmatrix}Y_\lambda=a^{i+1}b^{j-1}C_\lambda.$$
  Similarly, $\begin{pmatrix}  a & 0\\  0 &  b \end{pmatrix} B_\lambda=a^{i-1}b^{j+1}B_\lambda$; $\begin{pmatrix}  a & 0\\  0 &  b \end{pmatrix} X_\lambda=a^{i}b^{j}X_\lambda$; and $\begin{pmatrix}  a & 0\\  0 &  b \end{pmatrix} Y_\lambda=a^{i}b^{j}Y_\lambda$.
Correspondingly, the weights of $C_\lambda$, $B_\lambda$, $X_\lambda$ and $Y_\lambda$ are $\lambda+\beta$, $\lambda-\beta$, $\lambda$  and $\lambda$, respectively.

Further computations show
\begin{align}\label{eq: basic compu for sec 4-1}
&\begin{pmatrix}  a & m\\  0 &  b \end{pmatrix} B_\lambda=a^{i-1}b^{j+1}B_\lambda+ma^{i-1}b^{j}X_\lambda,\quad\begin{pmatrix}  a & m\\  0 &  b \end{pmatrix} C_\lambda=a^{i+1}b^{j-1}C_\lambda,\cr
&\begin{pmatrix}  a & m\\  0 &  b \end{pmatrix} X_\lambda=a^{i}b^{j}X_\lambda, \quad
\begin{pmatrix}  a & m\\  0 &  b \end{pmatrix} Y_\lambda=a^{i}b^{j}Y_\lambda+(i+j)ma^{i}b^{j-1}C_\lambda;
\end{align}
and
\begin{align}\label{eq: basic compu for sec 4-2}
&\begin{pmatrix}  a & 0\\  n &  b \end{pmatrix} B_\lambda=a^{i-1}b^{j+1}B_\lambda,\quad
\begin{pmatrix}  a & 0\\  n &  b \end{pmatrix} C_\lambda=a^{i+1}b^{j-1}C_\lambda+na^{i}b^{j-1} Y_\lambda,\cr
&\begin{pmatrix}   a & 0\\  n &  b \end{pmatrix} X_\lambda=a^{i}b^{j}X_\lambda+(i+j)na^{i-1}b^{j}B_\lambda,\quad \begin{pmatrix}  a & 0\\  n  &b \end{pmatrix} Y_\lambda=a^{i}b^{j}Y_\lambda.
\end{align}

\subsection{}\label{3.5} Summarizing up, we have the following lemma.
 \begin{lemma}\label{lem: basic computation sec 4} The following statements hold.
  \begin{itemize}
  \item[(1)] The induced modules  $H_{\beta}^0(\lambda)=\bk C_\lambda+\bk Y_\lambda$; and $H_{-\beta}^0(\lambda+\beta)=\bk B_{\lambda+\beta}+\bk X_{\lambda+\beta}$.

   \item[(2)] Furthermore,
  \begin{itemize}
  \item[(2.1)] when $p=0$ or $p\nmid |\lambda|$, $H_{\beta}^0(\lambda+\beta)$ and $H_{-\beta}^0(\lambda)$ are mutually isomorphic. Both of them are irreducible.
  \item[(2.2)]  When $p\ne 0$ and $p\mid |\lambda|$, then $X_\lambda=Y_\lambda$. Consequently, $L_{B_{-\beta}}(\lambda)\cong L_{B_{\beta}}(\lambda)\cong \textbf{k}X_{\lambda}$, and
      \begin{align}\label{eq: SES for Sec 4}
   H^0_{\beta}(\lambda)/ L_{B_{\beta}}(\lambda)\cong
   (\textbf{k}C_\lambda+\textbf{k}X_\lambda)/\textbf{k}X_\lambda
   \cong L_{B_{\beta}}(\lambda+\beta)
   \end{align}
along with
\begin{align*}\label{eq: SES1 for Sec 4}
   H^0_{-\beta}(\lambda+\beta)/ L_{B_{-\beta}}(\lambda+\beta)&\cong
   (\textbf{k}B_{\lambda+\beta}+\textbf{k}
   X_{\lambda+\beta})/\textbf{k}X_{\lambda+\beta}
   \cong L_{B_{-\beta}}(\lambda).
   \end{align*}
   \end{itemize}
    \item[(3)] When $p\ne 0$ with $p\mid |\lambda|$, there is a homomorphism
       \begin{equation}\label{eq: Upsilon}
        \Upsilon_{\lambda,\beta}: H_{-\beta}^0(\lambda+\beta)\longrightarrow H_{-\beta}^0(\lambda)
        \end{equation}
        satisfying $\texttt{im}(\Upsilon_{\lambda,\beta})\cong \textbf{k}X_\lambda\cong L_{B_{\beta}}(\lambda)$ and $\texttt{ker}(\Upsilon_{\lambda,\beta})\cong \textbf{k}X_{\lambda+\beta}\cong L_{B_{\beta}}(\lambda+\beta)$.
  \end{itemize}
  \end{lemma}

   \begin{proof}
       (1) and (2) follows from  \cite[\S7.3]{MZ}.

   (3) As  $H^0_{-\beta}(\lambda+\beta)$ (resp. $H^0_{-\beta}(\lambda)$) admits a $\textbf{k}$-basis consisting of  $B_{\lambda+\beta}$ and $X_{\lambda+\beta}$ (resp. $B_{\lambda}$ and $X_{\lambda}$), we can define  a map $\Upsilon_\beta$ from $ H^0_{-\beta}(\lambda+\beta)$ to $ H^0_{-\beta}(\lambda)$ via  $\Upsilon_\beta(B_{\lambda+\beta})=X_{\lambda}$ and  $\Upsilon_\beta(X_{\lambda+\beta})=0$.
 From \S\ref{3.4}, it follows that $\Upsilon_\beta$
    is really a $G$-module homomorphism   satisfying $\texttt{im}(\Upsilon_\beta)\cong \textbf{k}X_\lambda\cong L_{B_{\beta}}(\lambda)$ and $\texttt{ker}(\Upsilon_\beta)\cong \textbf{k}X_{\lambda+\beta}\cong L_{B_{\beta}}(\lambda+\beta)$.
   \end{proof}

\subsection{}\label{sec: conv-2}  From now on, we always take $A$ to be a principal ideal domain (PID for short).  Denote by $\bbk$ the fractional field of $A$.

As in \S\ref{sec: conv}(1),  let $G_A=\GL_A(1|1)$. In the same sense,  we can talk about $(B_{\pm\beta})_A$. More generally, for a closed subgroup $H$ of $G$ we can talk about $H_A$ as long as $H$ can be defined over $\bbz$.
Then one can define the induced modules over $A$
$$H_{\pm\beta,A}^0(A_\lambda):=\ind^{G_A}_{(B_{\pm\beta})_A}(A_\lambda).$$
Here and further  the  notation $A_\lambda$ indicates the rank-one $(B_{\pm\beta})_A$-module of weight $\lambda$ over $A$.
By the same arguments, an analogue of Lemma \ref{lem: basic computation sec 4} yields
$$H_{\pm\beta,A}^0(-)=H_{\pm\beta,\bbz}^0(-)\otimes_\bbz A$$
with $H_{\pm\beta}^0(-)=H_{\pm\beta,\bbz}^0(-)\otimes_\bbz \bk$.

Let $\lambda=(i|j)$. According to \S \ref{3.4}-\ref{3.5}, we can define the following homomorphisms $$T_{A_\lambda,\beta}: H_{-\beta,A}^0(\lambda+\beta )\rightarrow H_{\beta,A}^0(\lambda)$$ 
   with \begin{equation}\label{eq4.1}
 T_{A_\lambda,\beta}(v)=\begin{cases}
      Y_{\lambda}, & v=B_{\lambda+\beta}\\
      (i+j)C_{\lambda},  &  v=X_{\lambda+\beta}.
    \end{cases}
\end{equation}

And

 $$T'_{A_\lambda,\beta}: H_{\beta,A}^0(\lambda )\rightarrow H_{-\beta,A}^0(\lambda+\beta )$$ 
   with \begin{equation}
 T'_{A_\lambda,\beta}(v)=\begin{cases}
      X_{\lambda+\beta}, & v=C_\lambda\\
      (i+j)B_{\lambda+\beta},  &  v=Y_\lambda.
    \end{cases}
\end{equation}

\section{Induced modules: Case  $\GL_A(m|n)$}\label{S6S}
   From now on we always suppose $G=\GL(m|n)$. In this section, we  keep the notations and assumptions in \S\ref{sec: conv-2}
   for some Levi subgroups of $G$ isomorphic to $\GL(1|1)$. In particular, $A$ is a given PID, and $\bbk$ is the fractional field of $A$. Suppose $\beta$ is a given odd root which is naturally isotropic, i.e. $(\beta, \beta)=0$.

\subsection{} \label{sec: mini para} Suppose $K^+$ and $K_{-\beta}^+$ are a given pair of adjacent Borel subgroups, the latter of which are produced by an odd reflection $\hat r_\beta$ from the former as in \S\ref{2.2.1}.

 We uniformly write $K^+_{\beta}$ for $K^+$. Then the purely-even groups of $K^+_{\pm\beta}$ are just $B_\ev^+$. There is a minimal parabolic subgroup $P^+(\beta)$ of $G$ containing $K^+_{\pm\beta}$.  The opposite Borels are denoted by $K_{\mp\beta}$ $(=K^-_{\pm\beta})$ respectively, this is to say, $K_{-\beta}$ is opposite to $K_\beta^+$, and $K_{\beta}$ is opposite to $K_{-\beta}^+$. The opposite parabolic subgroup is denoted by $P(\beta)$ (=$P^-(\beta)$).

 Set $\frak{k}^+_\pm=\texttt{Lie}(K^+_{\pm\beta})$, and write $\frak{k}^+=\hhh+\sum_{\gamma\in S}\ggg_\gamma$ with $S=S_\bz\cup S_\bo$ being a closed subset of $\Phi$ corresponding to $K^+_{\beta}$. Then $\frak{k}^+_-=\hhh+\sum_{\gamma\in S_-}\ggg_\gamma$ with $S_-:=S\backslash\{\beta\}\cup\{-\beta\}$. Precisely, $S_\bz=(S_-)_\bz=\Phi^+_\bz$ for the even roots,  and  $$(S_-)_\bo=S_\bo\backslash\{\beta\}\cup\{-\beta\} \text{ for the odd roots}.$$
 Similarly, we can describe $\frak{k}^-_\mp=\texttt{Lie}(K_{\mp\beta})$ with
 $\frak{k}^-_-=\hhh+\sum_{\gamma\in S}\ggg_{-\gamma}$
 and $\frak{k}^-_+=\hhh+\sum_{\gamma\in S_-}\ggg_{-\gamma}$.

Denote $K_{A,\mp\beta}=(K_{\mp\beta})_A$ and
$P_A(\beta)=P(\beta)_A$.
 As usual, we adopt the  notation $A_\lambda$ for the rank-one $K_{A,\mp\beta}$-module of weight $\lambda$ over $A$ if the context is clear. In particular, $(\lambda,\beta)=a+b$ where $a:=a_s$ and $b:=b_t$ with  $\lambda=\sum_{k=1}^ma_k\delta_k+\sum_{l=1}^nb_l\epsilon_l$ and  $\beta=\delta_s-\epsilon_t$.

 Keep in mind that $\beta$ is an odd root in $\GL(m|n)$.
  The parabolic subgroup $P(\beta)$ contains a Levi subgroup isomorphic to $\GL(1|1)$, which naturally maps to $P(\beta)\slash K_{\mp\beta}$ satisfying the condition in Theorem \ref{thm: zubakov}.   By this theorem,  there is an isomorphism of $\GL(1|1)$-modules
 \begin{align}\label{eq: red GL11 more}
 \texttt{ind}_{K_{A,\mp\beta}}^{P_A(\beta)}A_{\lambda}
 \cong \texttt{ind}_{\GL_A(1|1)\cap K_{A,\mp\beta}}^{\GL_A(1|1)}A_{\lambda}.
 \end{align}
 Thus, by Lemma \ref{lem: basic computation sec 4} the induced module  $\texttt{ind}_{K_{A,-\beta}}^{P_A(\beta)}A_{\lambda}$ (resp. $\texttt{ind}_{K_{A,\beta}}^{P_A(\beta)}A_{\lambda}$) has an $A$-basis consisting of $B_{\lambda}$ and $X_{\lambda}$ (resp. $C_\lambda$ and $Y_\lambda$).
Owing to (\ref{eq: ind and ass}), we write $H^0_{\mp\beta}(\lambda)$ for  $\texttt{ind}_{K_{\mp\beta}}^{P(\beta)}\tbk_{\lambda}$ respectively
when working over  $\tbk$.

On the other hand, it turns out that the unipotent radical of $P_A(\beta)$ acts on $\texttt{ind}_{K_{A,\mp\beta}}^{P_A(\beta)}A_{\lambda}$  trivially (cf. \cite[Proposition 11.5]{Z1}). Similar to \S\ref{sec: conv-2}, we can define $P_A(\beta)$-module homomorphisms $$T_{A_\lambda,\beta}: \texttt{ind}_{K_{A,-\beta}}^{P_A(\beta)}A_{\lambda+\beta}\rightarrow \texttt{ind}_{K_{A,\beta}}^{P_A(\beta)}A_{\lambda}$$
   with \begin{equation}\label{eq4.1'}
 T_{A_\lambda,\beta}(v)=\begin{cases}
      Y_{\lambda}, &\text{when } v=B_{\lambda+\beta}\\
      (a+b)C_{\lambda},  & \text{when } v=X_{\lambda+\beta}
    \end{cases}
\end{equation}
and
 $$T'_{A_\lambda,\beta}: \texttt{ind}_{K_{A,\beta}}^{P_A(\beta)}A_{\lambda}\rightarrow \texttt{ind}_{K_{A,-\beta}}^{P_A(\beta)}A_{\lambda+\beta}$$ 
   with \begin{equation}\label{eq4.1''}
 T'_{A_\lambda,\alpha}(v)=\begin{cases}
      X_{\lambda+\beta}, &\text{when } v=C_\lambda\\
      (a+b)B_{\lambda+\beta},  &\text{when }  v=Y_\lambda.
    \end{cases}
\end{equation}

Then we have the following lemma.
\begin{lemma}\label{lem: const composite for odd} Keep the notations as above.
 Both $T_{A_\lambda,\beta}\circ T'_{A_\lambda,\beta}$ and $T'_{A_\lambda,\beta}\circ T_{A_\lambda,\beta}$ are multiplication by $a+b$. Furthermore, when $(\lambda, \beta)=a+b\neq0$ in $\bbk$, $T_{A_\lambda,\beta}$ (resp. $T'_{A_\lambda,\beta}$) is injective and  $T_{\bbk_\lambda,\beta}=T_{A_\lambda,\beta}\otimes \bbk$ (resp. $T'_{\bbk_\lambda,\beta}=T'_{A_\lambda,\beta}\otimes \bbk$) is an isomorphism.
 \end{lemma}

\subsection{}\label{4.9}
 As in the case of ordinary algebraic groups, the following statements still hold (see for example \cite[Proposition 11.3]{Z1}):
  \begin{equation}\label{eq6.0}H^i(G_A/K_{A,\mp\beta},\lambda)
   \cong R^{i}\texttt{ind}^{G_A}_{P_A(\beta)}
  (\texttt{ind}_{K_{A,\mp\beta}}^{P_A(\beta)}A_\lambda).\end{equation}
 So we can apply the functor $\texttt{ind}^{G_A}_{P_A(\alpha)}(-)$ to the maps $T_{A_\lambda,\beta}$ and $T'_{A_\lambda,\beta}$ (see \S\ref{sec: conv-2}). Then we have two homomorphisms
 $$\widetilde T_{A_\lambda,\beta}: H^0(G_A/K_{A,-\beta}, {\lambda+\beta})\rightarrow H^0(G_A/K_{A,\beta}, \lambda)$$
   and
$$\widetilde T'_{A_\lambda,\beta}: H^0(G_A/K_{A,\beta}, \lambda )\rightarrow H^0(G_A/K_{A,-\beta}, {\lambda+\beta}).$$
satisfying  that
 \begin{align}\label{eq: scale aplusb}
 \text{both } \widetilde T_{A_\lambda,\beta}\circ \widetilde T'_{A_\lambda,\beta} \text{ and }\widetilde T'_{A_\lambda,\beta}\circ\widetilde T_{A_\lambda,\beta} \text{ are multiplication by }  a+b.
   \end{align}
   Consequently,   by Lemma \ref{lem: const composite for odd} both $\widetilde T_{A_\lambda,\beta}$ and  $\widetilde T'_{A_\lambda,\beta}$ are injective if $a+b\ne 0$ in $\bbk$. Furthermore,  $\widetilde T_{\bbk,\beta}=\widetilde T_{A_\lambda,\beta}\otimes \bbk$ (resp. $\widetilde T'_{\bbk,\beta}=\widetilde T'_{A_\lambda,\beta}\otimes \bbk$) is an isomorphism whenever $(\lambda, \beta)=a+b\neq0$.

\subsection{}  Similarly we can define as above
$$\tilde \Upsilon_{\tbk_\lambda,\beta}: H^0(G/K_{-\beta}, {\lambda+\beta})\rightarrow H^0(G/K_{-\beta}, \lambda)$$
via applying the functor $\texttt{ind}^{G}_{P(\alpha)}(-)$ to the homomorphism $\Upsilon_{\lambda,\beta}$ defined in (\ref{eq: Upsilon}).

\subsection{}\label{4.8s}
Denote by  $L_{K_{\mp\beta}}(\lambda)$ the socle of $H^0(G/K_{\mp\beta}, {\lambda})$.
We have the following lemma (compatible with  \cite[Lemma 5.2]{PS} in the case of Lie superalgebras).

\begin{lemma}\label{L4.1} 
The following statements hold.

\begin{itemize}
\item[(1)] When  $(\lambda,\beta)\not\equiv0\mod p$,   both $\widetilde T_{\textbf{k}_\lambda,\beta}$ and $\widetilde T'_{\textbf{k}_\lambda,\beta}$ are isomorphisms. And $L_{K_{-\beta}}(\lambda+\beta)\cong L_{K_\beta}(\lambda)$.

\item[(2)]
 When  $(\lambda,\beta)\equiv0\mod p$,
 \begin{itemize}
 \item[(2.a)] (2.a.1) $\texttt{ker}(\widetilde T_{\bk_\lambda,\beta})=\texttt{im}(\widetilde  T'_{\bk_\lambda,\beta})\cong\texttt{coker}(\widetilde T_{\bk_\lambda,\beta})$;

(2.a.2) $\texttt{ker}(\widetilde T'_{\bk_\lambda,\beta})=\texttt{im}(\widetilde T_{\bk_\lambda,\beta})
\cong\texttt{coker}(\widetilde T'_{\bk_\lambda,\beta})$.

 \item[(2.b)]
 Furthermore, $L_{K_{-\beta}}(\lambda)\cong L_{K_\beta}(\lambda)$.
\item[(2.c)]
$\texttt{im}(\tilde\Upsilon_{\bk_\lambda, i\beta})=\texttt{ker}(\tilde\Upsilon_{\bk_\lambda,  (i-1)\beta})$ for all positive integers $i$.
 \end{itemize}
\end{itemize}
\end{lemma}

\begin{proof}
(1) Suppose $(\lambda,\beta)\not\equiv 0\mod p$. The statement follows from (\ref{eq: scale aplusb}).

(2)  Suppose $(\lambda,\beta)\equiv0\mod p$. For (2.a), keep in mind that both $K_{\pm\beta}$ have purely-even group $B_\ev$.
Hence
by Theorem \ref{thm: Shi thm}(2), we have isomorphisms as $T$-modules
\begin{align}\label{eq: decomps}
 &H^0(G/K_{-\beta}, {\lambda+\beta})\cong  H^0(G_\ev\slash B_\ev,\lambda+\beta)\otimes {\bigwedge}^\bullet(\sum_{\gamma\in S_1}\ggg_{\gamma})^* \quad{ and }\cr
 &H^0(G/K_{\beta}, \lambda)\cong H^0(G_\ev\slash B_\ev,\lambda)\otimes {\bigwedge}^\bullet(\sum_{\gamma\in (S_-)_1}\ggg_{\gamma})^*.
 \end{align}
 By the representations theory of algebraic groups (see \cite[Proposition II.2.2]{Jan3},  $H^0(G_\ev\slash B_\ev,\lambda)$ (resp.  $H^0(G_\ev\slash B_\ev,\lambda+\beta)$) has one-dimensional weight space of the $B_\ev^+$-highest weight $\lambda$ (resp. $\lambda+\beta$). In the sense of (\ref{eq6.0}) along with (\ref{eq4.1'}) and (\ref{eq4.1''}), the height weight spaces are clearly spanned by $Y_\lambda$ and $X_{\lambda+\beta}$, respectively.
On the other hand, by the arguments in \S\ref{sec: mini para} we know $S_\bz=(S_-)_\bz=\Phi_\bz^+$. Consequently,
 the even homogenous spaces of the second parts in the tensor products of the $T$-module decomposition formula (\ref{eq: decomps}) are all of negative roots, i.e. in $\Phi^-_\bz$. Hence the $B_\ev^+$-highest weight space of $H^0(G_\ev\slash B_\ev, \diamondsuit)$ remains the ones of $H^0(G/K_{\beta},\diamondsuit)$ for $\diamondsuit\in \{\lambda,\lambda+\beta\}$. Hence,  (\ref{eq4.1'}) and (\ref{eq4.1''}) are still true for $\widetilde T_{A_\lambda,\beta}$ and $\widetilde T'_{A_\lambda,\beta}$, respectively.

 Under the assumption $(\lambda, \beta)\equiv0\mod p$, by (\ref{eq4.1'}) and (\ref{eq4.1''}) we have $\texttt{ker}(T_{\bk_\lambda,\beta})=\bk X_{\lambda+\beta}$ and $\texttt{im}(T_{\bk_\lambda,\beta})=\bk Y_\lambda$, and then $\texttt{coker}(T_{\bk_\lambda,\beta})\cong \bk X_{\lambda+\beta}=\texttt{ker}(T_{\bk_\lambda,\beta})$. Similarly, $\texttt{ker}(T'_{\bk_\lambda,\beta})=\bk Y_{\lambda}$ and $\texttt{im}(T'_{\bk_\lambda,\beta})=\bk X_{\lambda+\beta}$, and then $\texttt{coker}(T'_{\bk_\lambda,\beta})\cong \bk Y_{\lambda}=\texttt{ker}(T'_{\bk_\lambda,\beta})$. Clearly, $\texttt{im}(T_{\bk_\lambda,\beta})=\texttt{ker}(T'_{\bk_\lambda,\beta})$,
  and $\texttt{im}(T'_{\bk_\lambda,\beta})=\texttt{ker}(T_{\bk_\lambda,\beta})$.

On the other hand, by  definition  $\tilde T_{\bk_\lambda,\beta}=\texttt{ind}^{G}_{P(\beta)}(T_{\bk_\lambda,\beta})$,
and  $\tilde T'_{\bk_\lambda,\beta}=\texttt{ind}^{G}_{P(\beta)}(T'_{\bk_\lambda,\beta})$.
So  $\texttt{im} (\tilde T_{\bk_\lambda,\beta})=
\texttt{ind}^{G}_{P(\beta)}(\texttt{im}(T_{\bk_\lambda,\beta}))$,
$\texttt{ind}^{G}_{P(\beta)}(\texttt{ker}(T_{\bk_\lambda,\beta}))\subset \texttt{ker}(\tilde T_{\bk_\lambda,\beta})$.
On the other hand, the functor $\texttt{ind}^{G}_{P(\beta)}(-)$
is left exact.
Hence  $\texttt{ind}^{G}_{P(\beta)}(-)$ preserves the exact sequence $$0\rightarrow\texttt{ker}(T_{\bk_\lambda,\beta})\rightarrow \texttt{ind}^{P(\beta)}_{K_{-\beta}}(\bk_{\lambda+\beta})\rightarrow \texttt{im} (T_{\bk_\lambda,\beta})\rightarrow 0.$$
The similar result also holds for $\tilde T'_{\bk_\lambda,\beta}$. Combining with the above arguments,  the two statements in (2.a) are proved.

We now check (2.b). Keep (\ref{eq: red GL11 more}) in mind.
 Turn back to the results in the case of $\GL(1|1)$.
Recall when $(\lambda,\beta)\equiv0\mod p$, $H^0_{\beta}(\lambda)$ (resp. $H^0_{-\beta}(\lambda+\beta)$) has one-dimensional socle $L_\beta(\lambda)$ (resp. $L_{-\beta}(\lambda+\beta)$) which is isomorphic to the  head of   $H^0_{-\beta}(\lambda+\beta)$ (resp. $H^0_{\beta}(\lambda)$) as $\GL(1|1)$-module in the sense of Lemma  \ref{lem: basic computation sec 4}. By the trivial action of the unipotent  subgroup of $P(\beta)$, $L_\beta(\lambda)$ can be extended to a $P(\beta)$-module.  Simply write $\texttt{ind}^{G}_{P(\beta)}(\lambda)$ for $\texttt{ind}^{G}_{P(\beta)}(L_\beta(\lambda))$. By (\ref{eq: SES for Sec 4}), there is an short exact sequence of $P(\beta)$-modules
$$ 0\rightarrow L_{\beta}(\lambda)\rightarrow H^0_{\beta}(\lambda)\rightarrow  L_{\beta}(\lambda+\beta)\rightarrow 0.  $$
Applying the      {left}  exact functor  $\texttt{ind}_{P(\beta)}^{G}(-)$, we have
  \begin{equation}\label{eqq6.4}0\rightarrow \texttt{ind}_{P(\beta)}^{G}(\lambda)\rightarrow H^0(G/K_\beta, {\lambda})\rightarrow \texttt{ind}_{P(\beta)}^{G}(\lambda+\beta)\rightarrow R^1\texttt{ind}_{P(\beta)}^{G}(\lambda)\rightarrow \cdots. \end{equation}
Thus, it follows that the socle of $\texttt{ind}_{P(\beta)}^{G}(\lambda)$ is isomorphic to $L_{K_\beta}(\lambda)$.

Similarly, we have
\begin{equation}\label{eqq6.5}0\rightarrow \texttt{ind}_{P(\beta)}^{G}(\lambda)\rightarrow H^0(G/K_{-\beta}, \lambda)\rightarrow \texttt{ind}_{P(\beta)}^{G}(\lambda-\beta)\rightarrow R^1\texttt{ind}_{P(\beta)}^{G}(\lambda)\rightarrow \cdots
\end{equation}
 This implies  the socle of $\texttt{ind}_{P(\beta)}^{G}(\lambda)$ is isomorphic to $L_{K_{-\beta}}(\lambda)$.
It is proved that  $L_{K_{-\beta}}(\lambda)\cong L_{K_{\beta}}(\lambda)$.

As to  (2.c),   $(\lambda,\beta)\equiv0\mod p$ implies   $(\lambda+i\beta,\beta)\equiv0\mod p, \forall i\in \mathbb N$. Similar to Lemma \ref{lem: basic computation sec 4}(3), there exists $\Upsilon_{\lambda, i\beta}: H^0_{-\beta}(\lambda+i\beta)\rightarrow H^0_{-\beta}(\lambda+(i-1)\beta), \forall i\in \mathbb N$ such that $\texttt{im}(\Upsilon_{\lambda, i\beta})\cong \textbf{k}X_{\lambda+(i-1)\beta}$ and $\texttt{ker}(\Upsilon_{\lambda, i\beta})\cong \textbf{k}X_{\lambda+i\beta}$. Then $\texttt{im}(\Upsilon_{\lambda,i\beta})\cong \textbf{k}X_{\lambda+(i-1)\beta}\cong \texttt{ker}(\Upsilon_{\lambda,(i-1)\beta})$. In particular, $\texttt{im}(\Upsilon_{\lambda,2\beta})\cong \textbf{k}X_{\lambda+\beta}\cong \texttt{ker}(T_{\bk_{\lambda},\beta})$.

Define  $$\tilde\Upsilon_{\bk_{\lambda}, i\beta}: H^0(G\slash K_{-\beta}, \lambda+i\beta)\rightarrow H^0(G\slash K_{-\beta}, \lambda+(i-1)\beta)$$ via $\tilde \Upsilon_{\bk_{\lambda},i\beta}=\texttt{ind}^{G}_{P(\beta)}(\Upsilon_{\lambda, i\beta})$.
Then
  we have
  \begin{align}\label{eq: ind inducing iso}
  \texttt{ind}^{G}_{P(\beta)}(\texttt{im}(\Upsilon_{\lambda, i\beta}))\cong \texttt{ind}^{G}_{P(\beta)}(\texttt{ker}(\Upsilon_{\lambda, (i-1)\beta})).
  \end{align}
   Keep it in mind that
$$\texttt{ind}^{G}_{P(\beta)}(\texttt{ker}(\Upsilon_{\lambda, i\beta}))\subseteq\texttt{ker}(\tilde\Upsilon_{\bk_{\lambda}, i\beta})\text{  and }\texttt{im}(\tilde\Upsilon_{\bk_{\lambda}, i\beta})=\texttt{ind}^{G}_{P(\beta)}(\texttt{im}(\Upsilon_{\lambda, i\beta}))$$
for all $i$. Hence, the left exactness of $\tilde\Upsilon_{\lambda,i\beta}(-)$ ensures that it preserves the following short exact sequence coming from Lemma \ref{lem: basic computation sec 4}(3)
$$0\rightarrow\texttt{ker}(\Upsilon_{\bk_\lambda,i\beta})\rightarrow \texttt{ind}^{P(\beta)}_{K_{-\beta}}(\bk_{\lambda+i\beta})\rightarrow \texttt{im} (\Upsilon_{\bk_\lambda,i\beta})\rightarrow 0. $$
 Hence, we have $\texttt{ind}^{G}_{P(\beta)}(\texttt{ker}(\Upsilon_{\lambda, i\beta}))=\texttt{ker}(\tilde\Upsilon_{\bk_{\lambda}, i\beta})$.
 In summary, (\ref{eq: ind inducing iso}) gives rise to
 the desired equality
$\texttt{im}(\tilde\Upsilon_{\bk_\lambda, i\beta})=\texttt{ker}(\tilde\Upsilon_{\bk_\lambda,  (i-1)\beta})$.

The proof is completed.
\end{proof}

As a consequence, we have
\begin{corollary}\label{cor: odd grod pre} Let $\lambda\in X^+(T)$ satisfying $(\lambda,\gamma)\ne 0$ for any odd root $\gamma$.  If $\beta$ is an odd root with $(\lambda,\beta)\equiv0\mod p$, the following equation holds in the Grothendieck group of $G$-module category
\begin{align}\label{eq: Groth eq}
[\texttt{coker}\tilde T_{\bk_\lambda,\beta}]=[H^0(G\slash K_{\beta}, \lambda)]+\sum_{k=1}^\infty(-1)^k [H^0(G\slash K_{-\beta}, \lambda+k\beta)].
\end{align}
\end{corollary}

\begin{proof} Note that all $\lambda+k\beta$ still satisfy $(\lambda+k\beta,\beta)\ne 0$ but $(\lambda+k\beta,\beta)\equiv0\mod p$, and consequently they all lie in $X^+(T)$ (see the forthcoming Lemma \ref{lem: typical} where more general notions related will be introduced).  According to  Lemma \ref{L4.1}(2),  there is a long exact sequence
\begin{align}\label{eq: long exact seq}
\cdots {\overset{\tilde \Upsilon_{\bk_{\lambda, 4\beta}}}\longrightarrow} H^0(G\slash K_{-\beta}, \lambda+3\beta){\overset{\tilde \Upsilon_{\bk_{\lambda, 3\beta}}}\longrightarrow}
H^0(G\slash K_{-\beta}, \lambda+2\beta){\overset{\tilde \Upsilon_{\bk_{\lambda, 2\beta}}}\longrightarrow}\cr
 H^0(G\slash K_{-\beta}, \lambda+\beta){\overset{\tilde T_{\bk_{\lambda},\beta}}\longrightarrow}
H^0(G\slash K_{\beta}, \lambda){\overset{\tilde T'_{\bk_\lambda,\beta}}\longrightarrow}
\texttt{im}\tilde T'_{\bk_\lambda,\beta}
\cong \texttt{coker}\tilde T_{\bk_\lambda,\beta}\longrightarrow 0.
\end{align}
The formula  (\ref{eq: Groth eq}) follows.
\end{proof}

\subsection{Characters}\label{sec: character} For any finite-dimensional $T$-module $M$, $M$ can be decomposed into a sum of weight spaces as $M=\sum_{\mu\in X(T)}M_\mu$.
As usual, we adopt the formal character of $M$, that is, $\texttt{ch}(M)=\sum_{\mu\in X(T)}\texttt{dim}(M_\mu) e^\mu$ in $\bbz[X(T)]$ where $\{e^\mu\}$ with $\mu$ running through $X(T)$ are the standard basis of the group ring $\bbz[X(T)]$ over $\bbz$, satisfying $e^{\mu_1}e^{\mu_2}=e^{\mu_1+\mu_2}$.   For example, one has for $\lambda\in X^+(T)$,
$$\texttt{ch}(H^0(\lambda))={A(\lambda+\rho_\bz)\over A(\rho_\bz)}\Xi$$
 where $A(\mu):=\sum_{w\in W}(-1)^{l(w)}e^{w(\mu)}$ and $\Xi:=\prod_{\beta\in \Phi_{\bo}^+}(1+e^{-\beta})$ (see for example, \cite[Corollary 5.8]{Shi}). For finite-dimensional $G$-modules $M_i$ $(i=1,2)$ and $N$ one has
 $$\texttt{ch}(M_1\otimes M_2)=\texttt{ch}(M_1) \texttt{ch}(M_2)$$
 and has further
 $$\texttt{ch}(N)=\texttt{ch}(M_1)+\texttt{ch}(M_2)$$
 if there exists a short exact sequence of $G$-modules: $M_1\hookrightarrow N\twoheadrightarrow M_2$.

\section{Typical weights and Steinberg's tensor product theorem}

\subsection{Typical weights}
\begin{defn}\label{def: typical} A dominant weight $\lambda\in X^+(T):=\{\mu=\sum_{i=1}^ma_i\delta_i+\sum_{j=1}^nb_j\epsilon_j\in X(T)\mid a_1\geq a_2\geq\cdots \geq a_m; b_1\geq b_2\geq\cdots \geq b_n\}$ is called  typical if $(\lambda+\rho, \beta_i)\ne 0$ for all $i=1,\ldots,mn$. Otherwise, a dominant weight $\mu\in X^+(T)$ which is not typical is called atypical.
\end{defn}
 By a straightforward computation, it is readily verified that
 \begin{align}\label{eq: typical eq}
 (\lambda+\rho,\beta_i)=(\lambda_{i-1},\beta_i).
 \end{align}
 Hence $\lambda\in X^+(T)$ is typical if and only if $(\lambda_{i-1},\beta_i)\ne 0$ for all $i=1,\ldots,n$. The following observation is due to \cite{PeSe} or \cite{Se}.
 \begin{lemma}\label{lem: typical} Suppose  $\lambda\in X^+(T)$  is a typical weight. Then all $\lambda_i$ lie in $X^+(T)$.
 \end{lemma}

 \begin{proof} When $\lambda\in X^+(T)$ is typical, $\lambda_i$ is ${B^{(i)}}^+$-dominant. Note that the purely-even subgroup of ${B^{(i)}}^+$ is $B_\ev^+$. Consequently, $\lambda_i$ lies in $X^+(T)$.
 \end{proof}

By the above lemma,  Theorem \ref{thm: Shi thm}(2) along with the classical result \cite[Proposition II.2.6]{Jan3} on induced modules of reductive algebraic groups ensure that all  $H^0(G\slash B^{(i-1)}, \lambda_i)$ are  nonzero.

\begin{remark}\label{rem: complex typical} When considering representations of $\GL(m|n)$ over $\bbc$, one has that a typical weight $\lambda\in X^+(T)$ gives rise to a Kac module realization of  typical irreducible modules $L(\lambda)$ where a Kac module is an induced module from an irreducible $G_\ev$-module (see \cite{CW},  \cite{DT}, \cite{Kac},  \cite{PeSe1}, etc.). In this paper, we will present  a modular version of this result in the concluding section. What is the most important is that typical weights will give rise to Jantzen filtration and consequently a sum formula of characters, which is the main purpose of the present paper.
\end{remark}

\subsection{Steinberg's tensor product theorem and the reduction from atypicals to typicals}
There is a natural question: how to deal with atypical weights in characteristic $p>0$? in contrast with the privilege of typical weights mentioned in Remark \ref{rem: complex typical}. Steinberg's tensor product theorem will tell us that it is enough to understand typical irreducible characters along with even irreducible characters (for purely-even subgroups).

Set
\begin{align*}
X_p^+(T):=\{\sum_{i=1}^ma_i\delta_i+\sum_{j=1}^n b_j\epsilon_j\in X^+(T)\mid &0\leq(a_{i}-a_{i+1}), (b_{j}-b_{j+1})\leq p-1,\cr
 &\forall i=1,\ldots,m-1; j=1,\ldots,n-1\}.
\end{align*}
For any atypical weight $\mu\in X^+(T)$, we can write it in $p$-adic expression
$$\mu=\mu_0+p\mu_1+\cdots+p^r\mu_r$$
 such that all $\mu_i\in X_p^+(T)$. Denote by $\varpi_i$ the $i$th fundamental weight of $\GL(m)$ ($i=1,\ldots,m-1$) which means  $\varpi_i\in X^+(T)$ satisfying  $(\varpi_i,\delta_k-\delta_{k+1})=\delta_{ik}$ for $k=1,\ldots,m-1$.
  There exists great enough positive number $l$ ($>r$) such that $\nu:=p^l\sum_{i=1}^{m-1}\varpi_i$ satisfies that $\lambda:=\mu+\nu$ is a typical weight. Obviously, $\varpi:=\sum_{i=1}^{m-1}\varpi_i\in X_p^+(T)$.

  Owing to  Kujawa's work \cite{Ku}, we have the following Steinberg's tensor product theorem for $\GL(m|n)$.
 \begin{theorem}\label{thm: st ten prod} Keep the notation as above. There is an isomorphism of $G$-modules
 $$L(\lambda)\cong L(\mu)\otimes L_\ev(\varpi)^{[l]}$$
 where $L_\ev(\varpi)^{[l]}$ stands for  the $l$th Frobenius twist of $L_\ev(\varpi)$ (see \cite{Ku} for more details), and $L_\ev(\varpi)$ stands for the irreducible module of $G_\ev$ with highest weight $\varpi$.
 \end{theorem}

By \S\ref{sec: character}, this theorem ensures that the question of formulating  irreducible characters of atypical weights can be reduced to the study of irreducible characters of typical weights.

\section{Totally-odd induced modules and related homomorphisms}\label{S6}

Recall that we already have  a series of Borel subgroups $B^{(i)}$ ($1\leq i \leq mn$) in \S\ref{2.2.2}. We now introduce the totally-odd induced modules which have been mentioned in \S\ref{sec: totally early}.
\subsection{}
Keep in mind the notations and assumptions in \S \ref{2.2.1}-\S \ref{2.2.3} and \S\ref{S4.2}.  We have that the purely-even parts of $B^{(i)}, i=0,1, \ldots, mn$ are the same one $B_\ev$. For $1\leq i \leq mn$, let  $P(\beta_{i})$ is the minimal parabolic subgroups containing  $B^{(i-1)}$ and $B^{(i)}$.
 By induction on $i$, $i=1,\ldots,mn$ starting with  $\lambda_0:=\lambda$, set $\lambda_i=\lambda_{i-1}-\beta_i$.

  By the conventional notations in \S\ref{S4.2},
 the totally-odd induced modules are naturally introduced as below
 $$\total:=H^0(G\slash B^{(mn)}, \lambda_{mn}). $$

  From now on, we always suppose $\lambda$ is a given typical weight in $X^+(T)$.

\subsection{}\label{6.6}
On the other hand,  for $\mu\in X^+(T)$, inductively set for $i=1,\ldots, mn$  with $\mu^{(0)}:=\mu$,
 \begin{align}\label{eq: sequence from lambda}
 \mu^{(i)}=\begin{cases}\mu^{(i-1)}, \text{ if }(\lambda_i, \beta_i)\equiv 0 \mod p;\cr
 \mu^{(i-1)}-\beta_i, \text{ if }(\lambda_i, \beta_i)\not\equiv0\mod p.
 \end{cases}
  \end{align}

 Recall that irreducible $G$-modules coincides with finite-dimensional  irreducible integrable $\texttt{Dist}(G)$-modules (see \cite{BKu}). By \cite[Theorems 4.3 and 4.5]{BKu}, the terminal one $\mu^{(mn)}\in X^+(T)$ with $L_{B^{(mn)}}(\mu^{(mn)})\cong L_{B^{(0)}}(\mu^{(0)})$. This $L_{B^{(0)}}(\mu)$ is exactly $L(\mu)$ which is by definition the simple socle of $H^0(\mu)$.

More generally,  by Lemma \ref{L4.1} $$H^0(G/B^{(i-1)},{\mu})\cong H^0(G/B^{(i)},{\mu-\beta_i}) \text{ and } L_{B^{(i-1)}}({\mu})\cong L_{B^{(i)}}({\mu-\beta_i})$$
whenever  $(\mu,\beta_i)\not\equiv0\mod p$.
So we have
\begin{lemma} (cf. \cite[Lemma 4.2]{BKu})\label{lem: sequence of irr odd ref}
 Suppose $\mu\in X^+(T)$ is given.
  For $1\leq i \leq mn$,
\begin{equation}\label{ff7.2}
  L_{B^{(i-1)}}(\mu)\cong\begin{cases}
      L_{B^{(i)}}(\mu), & (\mu,\beta_i)\equiv0 \mod p,\\
     L_{B^{(i)}}(\mu-\beta_i),  & (\mu,\beta_i)\not\equiv0 \mod p.    \end{cases}
\end{equation}
\end{lemma}

\begin{remark}\label{rem: typical}

(1) By the above lemma, there exists weight $\tilde\mu\in X^+(T)$  such that \begin{equation}\label{eq: tilde iso}L(\mu)\cong L_{B^{(mn)}}(\tilde\mu).\end{equation}
Furthermore, this $\tilde\mu$ is exactly $\mu^{(mn)}$ (see  \cite[Theorem 4.3]{BKu}). By \cite[Theorem 4.5]{BKu},
\begin{align}\label{eq: tilde star iso}
L(\mu)^*\cong L(-w_\bz\tilde\mu).
\end{align}

(2)  A weight $\lambda\in X^+(T)$ will be called $p$-typical if $(\lambda_{i-1},\beta_i)\not\equiv 0 \mod p$ for all  $i\in\{1,\ldots, mn\}$.  Note that by (\ref{eq: typical eq}), $(\lambda_{i-1},\beta_i)\not\equiv 0 \mod p$ if and only if $(\lambda+\rho,\beta_i)\not\equiv 0 \mod p$. So the notion of $p$-typical weights here is  identical to the one of ``typical weights" introduced in \cite{ZS} and \cite[\S12]{Z1}. By Lemma \ref{L4.1} $H^0({\lambda})$ and $H^0(G/B^{(mn)},{\lambda_{mn}})$ are isomorphic if and only if $\lambda$ is $p$-typical.     {In this case $H^0(\lambda)$ is irreducible if additionally $\lambda$ lies in the fundamental alcove (see \cite[Proposition 12.10]{Z1}, \cite[Theorem 1]{Ma}), etc..}

\end{remark}

As a consequence, we have
\begin{corollary}\label{coro: total socle} For $\lambda\in X^+(T)$, the following statements hold.
\begin{itemize}
\item[(1)]  There exists $\mu\in X^+(T)$ such that $L_{B^{(mn)}}(\lambda_{mn})\cong L(\mu)$, equivalent to say,  $\lambda_{mn}=\tilde \mu$.
 \item[(2)] Set  $\gamma=-w_{\bar 0}\lambda+2\rho_{\bar 1}$. Then $L(\gamma)^*\cong L_{B^{(mn)}}(\lambda_{mn})$.
\end{itemize}
\end{corollary}

\begin{proof} (1) It follows from (\ref{eq: tilde iso}).

(2) Note that $L(\gamma)=L(-w_{\bar 0}(\lambda-2\rho_{\bar 1}))\cong L(-w_{\bar 0}\lambda_{mn})$, which is isomorphic  to $L(\mu)^*$ by (\ref{eq: tilde star iso}).  The desired isomorphism follows from (1).
\end{proof}

\subsection{}\label{S6.5'}
 Keep the notations and assumptions in \S\ref{sec: conv-2}.
As in \S\ref{3.5}, we  take $C_{\lambda_{i}}, Y_{\lambda_{i}}$ as a  $\textbf{k}$-basis of $H_{\beta_{i}}^0(\lambda_{i})$, and  take  $B_{\lambda_{i-1}}, X_{\lambda_{i-1}}$ as a $\textbf{k}$-basis of $H_{-\beta_{i}}^0(\lambda_{i-1})$.

Similar to (\ref{eq4.1'}), (\ref{eq4.1''}) and \S\ref{4.9}, we have $P_A(\beta_{i})$-module homomorphisms $$T_{A_{\lambda_{i}},\beta_{i}}: \texttt{ind}_{B_{A,-\beta_{i}}}^{P_A(\beta_{i})}A_{\lambda_{i-1}}\rightarrow \texttt{ind}_{B_{A,\beta_{i}}}^{P_A(\beta_{i})}A_{\lambda_{i}};$$ $$T'_{A_{\lambda_{i}},\beta_{i}}: \texttt{ind}_{B_{A,\beta_{i}}}^{P_A(\beta_{i})}A_{\lambda_{i}}\rightarrow \texttt{ind}_{B_{A,-\beta_{i}}}^{P_A(\beta_{i})}A_{\lambda_{i-1}}.$$

And  $G_A$-module homomorphisms
 $$\widetilde T_{A_{\lambda_{i}},\beta_{i}}: H^0(G_A/B_A^{(i-1)},\lambda_{i-1}) \rightarrow H^0(G_A/B_A^{(i)},\lambda_{i});$$
 $$\widetilde T'_{A_{\lambda_{i}},\beta_i}: H^0(G_A/B_A^{(i)},\lambda_{i})\rightarrow  H^0(G_A/B_A^{(i-1)},\lambda_{i-1}).$$

\subsection{} Keep the notations in \S\ref{sec: character}. Now we adopt the Euler characteristic for each finite dimensional $B$-module $M$,
$$\chi(M):=\sum_{i\geq0} (-1)^i \texttt{ch}H^i(M).$$ Then we  have the following facts.

\begin{prop}\label{lem: totally ind ch} Suppose $\lambda\in X^+(T)$.
The following statements hold.
\begin{itemize}
\item[(1)] $\texttt{ch}(\total)=\texttt{ch}(H^0(\lambda-2\rho_\bo))\Xi_{mn}$, where $\Xi_{mn}:=\prod_{\beta\in \Phi^+_\bo}(1+e^\beta)$.
\item[(2)] Set $\chi(\lambda)=\chi(\tbk_\lambda)$, and $\chi_\bz(\lambda):=\chi_\bz(\tbk_\lambda)$. Then
    \begin{itemize}
    \item[{\tiny(2.a)}]    $\chi(\lambda)=\chi_\bz(\lambda)\Xi$.
    \item[{\tiny(2.b)}] $\chi(w. \lambda)=\texttt{det}(w)\chi(\lambda)$ for any $w\in W$ with determinant $\texttt{det}(w)$.
        \end{itemize}
\end{itemize}
\end{prop}
\begin{proof}
(1) It follows from Theorem \ref{thm: Shi thm}(2).

(2) By Theorem \ref{thm: Shi thm}(2) again, for any $B$-module $M$ there is a $T$-module isomorphism
\begin{align*}
R^{i}\texttt{ind}^G_B(M)\cong R^{i}\texttt{ind}^{G_\ev}_{B_\ev}(M) \otimes \bwedge(\ggg\slash \bbb^-)^*_\bo.
\end{align*}
Note $\texttt{ch}( \bwedge(\ggg\slash \bbb^-)^*_\bo)=\Xi$.
The first part follows. As to the second one, note that $w(\rho_\bo)=\rho_\bo$.  It follows from the first part along with the result on  Euler characters of $G_\ev$-modules (see \cite[\S{II.5.9(1)}]{Jan3}).
\end{proof}

\subsection{Homomorphisms arising from odd reflections}\label{7.5}
According to \S \ref{S6.5'},  we have a sequence of homomorphisms
$$\CD
  H^0(G_A/B_A^{(0)},\lambda_{0}) @>\widetilde T_{A,\beta_{1}}>> H^0(G_A/B_A^{(1)},{\lambda_{1}}) @>\widetilde T_{A,\beta_{2}}>> \cdots
\endCD$$
$$\CD
  @>\widetilde T_{A,\beta_{mn}}>>H^0_{\texttt{total},A}(\lambda):= H^0(G_A/B_A^{(mn)},{\lambda_{mn}}).
\endCD$$
Denote the composite of these sequence of homomorphisms  by $\widetilde T_{A, {w}_{\bar 1}}$, i.e.  $$\widetilde T_{A,{w}_{\bar 1}}=\widetilde T_{A,\beta_{mn}}\circ \widetilde T_{A,\beta_{mn-1}}\circ\cdots\circ\widetilde T_{A,\beta_{1}}.$$
Similarly, consider  $\widetilde T'_{A,{w}_{\bar 1}}=\widetilde T'_{A,\beta_{1}}\circ \widetilde T'_{A,\beta_{2}}\circ\cdots\circ\widetilde T'_{A,\beta_{mn}}$.
Then we have
\begin{equation}\label{e6.5}\widetilde T_{A,{w}_{\bar 1}}:H^0(G_A/B_A^{(0)},\lambda_{0})\rightarrow H^0(G_A/B_A^{(mn)},{\lambda_{mn}}).\end{equation}
and
$$ \widetilde T'_{A,{w}_{\bar 1}}: H^0(G_A/B_A^{(mn)}, {\lambda_{mn}})\rightarrow H^0(G_A/B_A^{(0)}, {\lambda_{0}}).$$

Keep it in mind that $\bbk$ is the fractional field of $A$.
 Let $H_{\texttt{total},\bbk}^0(\lambda):=H^0_{\texttt{total},A}(\lambda)\otimes_A \bbk$.  Lemma \ref{lem: const composite for odd} yields the following result.
\begin{lemma}\label{lem: K-iso Odd}
 Suppose $\texttt{char}(\bbk)=0$, and $\lambda$ is typical. Then the natural extension of $T_{A,w_\bo}$
  \begin{equation}\label{eq: even comp K}\widetilde T_{K,w_{\bar1}}: H_\bbk^0(\lambda)  \rightarrow H^0_{\texttt{total},\bbk}(\lambda)
  \end{equation}
  is a $\bbk$-isomorphism.
  \end{lemma}

\subsection{} We turn to the case over $\tbk$.
 \subsubsection{} We first prove the following basic result. Recall that for an algebraic supergroup scheme and its module, one can talk about the fixed point space (cf. \cite[\S{I.2.10}]{Jan3} with change of any  commutative $\tbk$-algebra $A$ into any commutative $\tbk$-superalgebra $R$).     {One can talk about unipotent supergroup scheme  over $\tbk$. By a result of Masuoka, an affine group scheme $\textbf{U}$ is unipotent if and only if its purely-even subgroup scheme $\textbf{U}_\ev$ is unipotent (see \cite[Theorem 41]{Mas}, or \cite{ZU}).}
 Any irreducible module of a unipotent supergroup scheme  $\textbf{U}$ is one-dimensional and trivial (see \cite{Mas}, \cite{Z4}, etc.). Therefore, for a unipotent supergroup $\textbf{U}$,  its fixed-point space of a non-zero $\textbf{U}$-module is nonzero.

\begin{lemma}\label{lem: U fixed} Let ${U^{(mn)}}^+$ be the unipotent radical of ${B^{(mn)}}^+$. Then the ${U^{(mn)}}^+$-fixed point subspace $H^0(G\slash B^{(i-1)},\lambda_{i-1})^{{U^{(mn)}}^+}$  for any $i=1,\ldots,mn$ is  one-dimensional, which is exactly the weight space
$H^0(G\slash B^{(i-1)},\lambda_{i-1})_{\lambda_{mn}}$.
\end{lemma}

\begin{proof} Keep it in mind that the purely-even subgroup of ${B^{(mn)}}^+$ is $B_\ev^+$ which has unipotent radical $U_\ev^+$, and for any nonzero ${U^{(mn)}}^+$-module  $M$, the fixed-point space $M^{{U^{(mn)}}^+}$ must be nonzero. For simplicity of notations, we set $\calk=H^0(G\slash B^{(i-1)},\lambda_{i-1})$ which can be described  as
\begin{align*}
\calk=\{\phi\in \tbk[G]\mid &\phi(xb)=\lambda_{i-1}(b)^{-1}\phi(x)\text{ for any }x\in G(R), b\in B^{(i-1)}(R)\cr
 & \text{ and for any } R\in \sak\}.
 \end{align*}
The action of $G(R)$ is given by left translation.
For any $f\in \calk^{{U^{(mn)}}^+}$, it satisfies $f(ub)=b^{-1}f(1)$ for all $u\in {U^{(mn)}}^+(R)$ with $f(1)\in \tbk_{\lambda_{i-1}}$. Note that ${U^{(mn)}}^+$  and $B^{(i-1)}$  have purely-even subgroups identical to $U_\ev^+$ and $B_\ev$, respectively.
Hence, $f|_{U_\ev^+B_\ev}$ is determined by $f(1)$.
Recall that as a topological space, open subsets of $G$ are just the ones of $G_\ev$ while $U_\ev^+B_\ev$ is open dense (see \cite[\S8.3]{Sp} or \cite[\S{II.1.9}]{Jan3}). Then we can conclude that $f(1)$ determines  $f$. Correspondingly, $\calk^{{U^{(mn)}}^+}$ is one-dimensional. On the other hand, as mentioned previously $\calk^{{U^{(mn)}}^+}\ne 0$. So we have $\texttt{dim}\calk^{{U^{(mn)}}^+}=1$.

 According to Theorem \ref{thm: Shi thm}, $H^{0}(G\slash B^{(i-1)}, \lambda_{i-1})$ has a $G_\ev$-submodule isomorphic to $H^0_\ev(\lambda_{i-1})$ which  admits  one-dimensional weight space of the $B_\ev^+$-highest weight $\lambda_{i-1}=\lambda-(\beta_{1}+\ldots+\beta_{i-1})$.
 What's more,  by Theorem \ref{thm: Shi thm}(2) there is a $T$-module isomorphism
$$H^{0}(G\slash B^{(i-1)}, \lambda_{i-1}) \cong H^0_\ev(\lambda_{i-1}) \otimes \bwedge(\ggg\slash ({\bbb^{(i-1)}}^-))^*_\bo$$
 where ${\bbb^{(i-1)}}^-=\texttt{Lie}(B^{(i-1)})$.
It is readily known that the set of $T$-weights of
$(\ggg\slash {\bbb^{(i-1)}}^-)^*_\bo$ is exactly
$$\{-\beta_1,\cdots, -\beta_{mn-(i-1)}, \beta_{mn-(i-1)+1},\ldots,\beta_{mn}\}.$$
 Hence, $(\ggg\slash {\bbb^{(i-1)}}^-)^*_\bo$ has  one-dimensional weight space of the ${B^{(mn)}}^+$-highest weight $\beta_{i}+\beta_{i+1}+\ldots+\beta_{(mn)}$. Therefore,
 $H^0(G\slash B^{(i-1)},\lambda_{i-1})$ has one-dimensional weight space of the ${B^{(mn)}}^+$-highest weight $\lambda_{mn}=\lambda-\sum_{i=1}^{mn}\beta_i$.

 Combining the above, we have
 $$H^0(G\slash B^{(i-1)},\lambda_{i-1})^{{U^{(mn)}}^+}=H^0(G\slash B^{(i-1)},\lambda_{i-1})_{\lambda_{mn}}.$$
 The proof is completed.
\end{proof}

\subsubsection{} Now we are in a position to introduce the following important result.
\begin{prop}\label{prop: odd comp nonzero}
Let $\lambda\in X^+(T)$ be a typical weight.  The composite of the above homomorphisms
$$ \widetilde T_{\tbk,{w}_{\bar 1}}: H^0(\lambda)\rightarrow H^0_{\texttt{total}}(\lambda).$$
is nonzero.
\end{prop}

\begin{proof}   Applying Lemma \ref{L4.1} (its statement(1) and its statement (2.1.2)) to the case
$$\tilde T_{\tbk_{\lambda_{i}},\beta_i}:H^0(G\slash B^{(i-1)},\lambda_{i-1})\rightarrow H^0(G\slash B^{(i)},\lambda_{i}),$$
we have that $\texttt{im}(\tilde T_{\tbk,\beta_i})$ is not zero. Hence for the unipotent group  ${U^{(mn)}}^+$, the fixed point subspace $\texttt{im}(\tilde T_{\tbk,\beta_i})^{{U^{(mn)}}^+}$ is surely nonzero. By Lemma \ref{lem: U fixed}, $\texttt{im}(\tilde T_{\tbk,\beta_i})^{{U^{(mn)}}^+}=H^0(G\slash B^{(i)},\lambda_{i})_{\lambda_{mn}}$. Hence all one-dimensional ${U^{(mn)}}^+$-fixed point subspaces concerned are preserved by the sequence of  nonzero morphisms $\tilde T_{\tbk,\beta_i}$. Consequently, the composite of them are nonzero. The proof is completed.
\end{proof}

\begin{remark} Lemma \ref{lem: U fixed} and Proposition \ref{prop: odd comp nonzero} are true for the case of base fields of characteristic zero.
\end{remark}

\section{Weyl modules and related homomorphisms}
Keep the notations and assumptions as before. Especially, $w_{\bar 0}$ is the longest element of $W$. For  $\lambda\in X^+(T)$,  inductively set $\lambda_i$ with the initial one $\lambda_0=\lambda$, and the terminal one $\lambda_{mn}=\lambda-2\rho_\bo$. For $w\in W$, denote $w.\lambda=w(\lambda)-\rho$
 for $\rho=\rho_\bz-\rho_\bo=
{1\over 2}(\sum_{\alpha\in \Phi^+_{\bz}}\alpha-\sum_{\beta\in \Phi^+_{\bo}}\beta)$.

\subsection{Weyl modules} Recall the  length   of $w_{\bar 0}$ is $l(w_{\bar 0})=|\Phi_\bz^+|$.  For $\lambda\in X(T)$, define Weyl module
$$V(\lambda):= R^{l(w_{\bar 0})}\texttt{ind}^G_B(\tbk_{w_{\bar 0}.\lambda})\cong H^{l(w_{\bar 0})}(G/B, \mathscr{L}(\textbf k_{w_{\bar 0}.\lambda})).$$
As usual, we write $V(\lambda)= H^{l(w_{\bar 0})}({w_{\bar 0}.\lambda})$.
Furthermore,  it can be defined over any commutative $\bbz$-algebra $A$, written as
$$V_A(\lambda)=R^{l(w_{\bar 0})}\texttt{ind}^{G_A}_{B_A}(A_{w_\bz.\lambda}).$$
 Then $$V_A(\lambda)=
H_A^{l(w_{\bar 0})}(w_\bz.\lambda).  $$
Recall  we have (cf. Theorem \ref{Serre duality} and Lemma \ref{lem: Ber}) $$ (R^i\texttt{ind}^G_BM )^*\cong R^{l(w_\bz)-i}\texttt{ind}^G_B (M^*\otimes \texttt{Ber}(X))$$
 with $X=G\slash B$,  and
  $$\texttt{Ber}(X)\cong \scrL_\ev(\bwedge(\ggg\slash \bbb^-)_\bo\otimes \bk_{-2\rho_\bz}).$$
Hence,  we have    
$$V(\lambda)^*\cong \texttt{ind}^G_B(\textbf k_{w_{\bar 0}.\lambda}^*\otimes \texttt{Ber}(X)))\cong\texttt{ind}^G_H(\textbf k_{-w_{\bar 0}\lambda}\otimes \bwedge(\ggg\slash \bbb^-)_\bo).$$
Thus
 \begin{equation}\label{f8.1}V(\lambda)\cong \texttt{ind}^G_B(\textbf k_{-w_{\bar 0}\lambda}\otimes \bwedge(\ggg\slash \bbb^-)_\bo)^*.\end{equation}

  In the following arguments we denote $M=\textbf k_{-w_{\bar 0}\lambda}\otimes \bwedge(\ggg\slash \bbb^-)_\bo$ for the time being.

\begin{lemma} \label{lem9.1} Assume $\lambda\in X^+(T)$.
The following statements hold.
\begin{itemize}
\item[(1)] $H^{0}(G/B, \textbf k_{-w_{\bar 0}\lambda}\otimes \bwedge(\ggg\slash \bbb^-)_\bo)$ is nonzero.

\item[(2)] The module $H^{0}(G/B, \textbf k_{-w_{\bar 0}\lambda}\otimes \bwedge(\ggg\slash \bbb^-)_\bo)$ has $B^+$-highest weight $-w_{\bar 0}\lambda+2\rho_{\bar 1}$.

\item[(3)] Denote $\gamma=-w_{\bar 0}\lambda+2\rho_{\bar 1}$. Then
$H^{0}(G/B, \textbf k_{-w_{\bar 0}\lambda}\otimes \bwedge(\ggg\slash \bbb^-)_\bo)$  has simple socle isomorphic to $L(\gamma)$.
Correspondingly,  $V(\lambda)$ has simple head isomorphic to  $L(\gamma)^*\cong L(-w_0\tilde\gamma)$.

\item[(4)]  $\texttt{ch}(V(\lambda))=\texttt{ch}(H^0(\lambda))$.
\end{itemize}
\end{lemma}

\begin{proof} (1) Set $\calh:=H^{0}(G/B, \tbk_{-w_{\bar 0}\lambda}\otimes \bwedge(\ggg\slash \bbb^-)_\bo)$. By the above arguments, $\calh^*\cong H^{l(w_{\bar 0})}(G/B, {w_{\bar 0}.\lambda})$.  Note that as a $T$-module there is an isomorphism as below (see Theorem \ref{thm: Shi thm}(2))
\begin{align}\label{eq: Weyl mod decomp}
H^{l(w_{\bar 0})}({w_{\bar 0}.\lambda})\cong H_\ev^{l(w_{\bar 0})}(w_{\bar 0}.\lambda)\otimes \bwedge(\ggg\slash \bbb^-)^*_\bo.
\end{align}
By the classical Serre duality, $H^{l(w_{\bar 0})}(G_\ev/B_\ev, {w_{\bar 0}.\lambda})\cong H^0(G_\ev/B_\ev,{-w_\bz\lambda})^*\ne 0$. Hence Theorem \ref{thm: Shi thm}(1) ensures that $\calh$ is nonzero.

(2) We will prove that $\calh^{U^+}= \calh_{-w_{\bz}\lambda+2\rho_{\bo}}$ by the same argument as in the proof of Lemma \ref{lem: U fixed} with suitable  change (in the present situation, the associated sheaf over $G\slash B$ has rank greater than one).
In order to show this, by definition we have for any  commutative $\tbk$-superalgebra $R$
$$\calh=\{\phi\in \texttt{Mor}(G,M)\mid \phi(xb)=b^{-1}\phi(x)\text{ for all }x\in G(R), b\in B(R)\}$$
where $\texttt{Mor}(G,M)$ stands for the morphism set, $M$ is regarded as an additive supergroup scheme, i.e. $M(R)$ is identified with $M\otimes_\tbk R$.
The action of $G(R)$ is given by left translation. For any $f\in \calh^{U^+}$, it satisfies $f(ub)=b^{-1}f(1)$ for all $u\in U^+(R)$ where by definition $f(1)\in M(R)=R_{-w_{\bar 0}\lambda}\otimes \bwedge(\ggg\slash \bbb^-)_\bo$.
Hence, $f|_{U^+(R)B(R)}$ is determined by $f(1)$.
Recall again that as a topological space, open subsets of $G$ are by definition just ones of $G_\ev$ while $U_\ev^+B_\ev$ is open dense (see \cite[\S8.3]{Sp} or \cite[\S{II.1.9}]{Jan3}). Then we can conclude that $f(1)$ determines  $f$. Correspondingly, $\calh^{U^+}$ is of dimension one. On the other hand,  $\calh^{U^+}\ne 0$.
So we have $\texttt{dim}\calh^{U^+}=1$. Taking the generalized tensor identity into an account (see (\ref{eq: gen ten id}) and \cite[\S{I.4.8}]{Jan3}), by the above arguments
we have the following $T$-module decomposition
\begin{align}\label{eq: T-module decomp}
\calh
 \cong H^{0}_\ev(-w_{\bar 0}\lambda)\otimes \bwedge(\ggg\slash \bbb^-)_\bo.
 \end{align}
By comparing the weights, we know that $-w_{\bar 0}\lambda+2\rho_\bo$ is a $B^+$-maximal weight, and its weight space is one dimensional here, as below
\begin{align*}\calh_{-w_{\bar 0}\lambda+2\rho_\bo}
\cong&H_\ev^{0}({-w_{\bar 0}\lambda})_{-w_{\bar 0}\lambda}\otimes \wedge^{mn}(\ggg\slash \bbb^-)_\bo.\end{align*}

As in the case of reductive algebraic groups, if $\mu$ is a $B^+$-maximal one among the weights of $\calh$, then $\calh_\mu\subset \calh^{U^+}$. Hence  we can conclude
$\calh^{U^+}\cong H^{0}_\ev({-w_{\bar 0}\lambda})_{-w_{\bar 0}\lambda}\otimes \wedge^{mn}(\ggg\slash \bbb^-)_\bo$. Correspondingly, $\calh^{U^+}=\calh_{-w_\bz\lambda+2\rho_\bo}$,  and any weight of $\calh$ is smaller than $-w_\bz\lambda+2\rho_\bo$ in the sense of the standard positive root system.

(3) For the last statement, if $L_1$ and $L_2$ are two distinct simple submodules of $\calh$, then $L_1^{U^+}\oplus L_2^{U^+}\subset \calh^{U^+}$.
There would be a contradiction with  rank $1$ of $\calh^{U^+(R)}$ over $R$. Hence $H^{0}(G/B, \textbf k_{-w_{\bar 0}\lambda}\otimes \bwedge(\ggg\slash \bbb^-)_\bo)$ has simple socle $L(\gamma)$.
As $V(\lambda)\cong H^{0}(G/B, \textbf k_{-w_{\bar 0}\lambda}\otimes \bwedge(\ggg\slash \bbb^-)_\bo)^*$,  we have that $V(\lambda)$ has simple head $L(\gamma)^*$ which is isomorphic to $L(-w_0\tilde\gamma)$ by \cite[Theoremm 4.5]{BKu}.

(4) Note that $\texttt{ch}(V(\lambda))=\texttt{ch}(\mathcal{H})$.
This statement  is a consequence of (\ref{eq: gen coho decom}) and (\ref{eq: T-module decomp}).

The proof is completed.
\end{proof}

\subsection{}\label{S8.2}

Recall $L(\gamma)^*$ is isomorphic to the head of  $V(\lambda)$ (see Lemma \ref{lem9.1}). We have the following observation.
 \begin{lemma}\label{lem8.2}
 Keep the notations and assumptions as before. The head  of $V(\lambda)$ is isomorphic to the socle of $\total=H^0(G/B^{(mn)},\lambda_{mn})$.
\end{lemma}

\begin{remark}\label{r8.3}
Unlike the case of reductive algebraic groups,  in general, the head of Weyl module  $V(\lambda)$ is not isomorphic to $L(\lambda)$.
When $\lambda$ is $p$-typical (see Remark \ref{rem: typical}), by Lemma \ref{L4.1}(1), $L(\lambda)\cong L_{B^{(mn)}}(\lambda_{mn})$. On the other hand, in this case $\tilde\lambda=\lambda$. It follows that  $L(\lambda)\cong L_{B^{(mn)}}(\lambda_{mn})\cong L(\gamma)^*$. So in this case,  the head of $V(\lambda)$ is isomorphic to $L(\lambda)$.
\end{remark}

\subsection{Homomorphisms arising from Weyl groups}\label{sec: 8.3}
Now we will present  a homomorphism from $V(\lambda)$ to $H^0(\lambda)$, by exploiting the arguments for reductive algebraic groups (see \cite[\S{II.5} and \S{II.8}]{Jan3}) to our case.
\subsubsection{}
Let's first recall  some necessary structural information for $H^i(\lambda)$ over $\bbf$ which is  any given ground field. We temporarily suppose $G$ is over $\bbf$ for the following lemma.

 Let $\alpha$ be  a given simple even root with  even simple reflection $r_\alpha\in W$. Consider the minimal  parabolic subgroup $P(\alpha)$ containing $B=B^-$ and the purely-even root subgroup  $G_\alpha$ (see \cite[\S{II}.1.3]{Jan3}). We look  at the structure $R^i \texttt{ind}^{P(\alpha)}_B\bbf_\lambda$. Note that $\alpha$ is an even root. By Theorem \ref{thm: zubakov},  there is an isomorphism of $P(\alpha)_\ev$-modules: $R^i \texttt{ind}^{P(\alpha)}_B\bbf_\mu\cong
 R^i \texttt{ind}^{P(\alpha)_\ev}_{B_\ev}\bbf_\mu$ (with trivial $R_\alpha$-action on the left hand side for the unipotent radical $R_\alpha$ of $P(\alpha)$). So we can mimic \cite[Proposition II.5.2]{Jan3} as below, with a little modification.

 \begin{lemma}\label{lem: basic cohom comput} The following statements hold.
\begin{itemize}
\item[(1)] The unipotent radical of $P(\alpha)$ acts trivially on each $R^i \texttt{ind}^{P(\alpha)}_B\bbf_\lambda, \forall i\geq 0$.

\item[(2)] If $ (\lambda,\alpha)=-1$, then $R^\bullet \texttt{ind}^{P(\alpha)}_B\bbf_\lambda=0$.

\item[(3)]  \begin{itemize}\item[{\tiny(3.a)}] If $(\lambda,\alpha)=s \geq 0$, then $R^i \texttt{ind}^{P(\alpha)}_B\bbf_\lambda=0, \forall i\neq 0$ and $\texttt{ind}^{P(\alpha)}_B\bbf_\lambda$ has a basis $\{v_i\mid i=0,1,\ldots,s\}$ such that for all $i, 0\leq i \leq s$
    and any commutative  $\bbf$-superalgebra $R=R_\bz\oplus R_\bo$:
    \begin{itemize}
    \item[{\tiny(3.a.1)}]
    $tv_i=(\lambda-i\alpha)(t)v_i, \forall t\in T(R_\bz)$;
        \item[\tiny{(3.a.2)}] $x_\alpha(a)v_i=\sum\limits_{j=0}^{i}{i\choose j}a^{i-j}v_j, \forall a\in R_\bz$;
         \item[\tiny{(3.a.3)}] $x_{-\alpha}(a)v_i=\sum\limits_{j=i}^{s}{s-i\choose s-j}a^{j-i}v_j,  \forall a\in R_\bz$.
             \end{itemize}
 Here and later  $x_{\pm\alpha}(a)$ are  Chevalley generators of  Chevalley supergroups in the same sense of  reductive algebraic groups (see \cite[\S5.2]{FG12}, \cite[\S{II.1.19}]{Jan3}).
\item[{\tiny(3.b)}] If $(\lambda,\alpha)\leq -2$, then $R^i \texttt{ind}^{P(\alpha)}_B\bbf_{\lambda}=0, \forall i\neq 1$ and $R^1\texttt{ind}^{P(\alpha)}_B\bbf_{\lambda}$ has a basis $\{v'_i\mid i=0,1,\ldots,s\}$ such that for all $i, 0\leq i \leq s$ and each super commutative $\bbf$-superalgebra $R=R_\bz\oplus R_\bo$:
    \begin{itemize}
    \item[\tiny{(3.b.1)}] $tv'_i=(r_\alpha.\lambda-i\alpha)(t)v_i, \forall t\in T(R_\bz)$;
        \item[\tiny{(3.b.2)}] $x_\alpha(a)v'_i=\sum\limits_{j=0}^{i}{s-j\choose s-i}a^{i-j}v'_j, \forall a\in R_\bz$;
         \item[\tiny{(3.b.3)}] $x_{-\alpha}(a)v'_i=\sum\limits_{j=i}^{s}{j\choose i}a^{j-i}v'_j,  \forall a\in R_\bz.$
            \end{itemize}
\end{itemize}
\item[(4)] For any $P(\alpha)$-module $M$ and any $i\geq 0$, $R^i\texttt{ind}^G_B M\cong R^i\texttt{ind}^G_{P(\alpha)} M$.
\item[(5)]\begin{itemize}\item[{\tiny(5.a)}] If $(\lambda,\alpha)\geq 0$,  then $H^i(\lambda)\cong H^i(\texttt{ind}^{P(\alpha)}_B\lambda)$ for all $i$. Furthermore, if $\texttt{ch}(\bbf)=0$, or $\texttt{ch}(\bbf)=p>0$ with $(\lambda,\alpha)=rp^m-1, r,m\in\mathbb N, 0<r<p$, then for all $i$
$$H^{i+1}(r_\alpha.\lambda)\cong H^i(\lambda).$$
\item[{\tiny(5.b)}] If $(\lambda,\alpha)\leq -2$, then for all $i$
$$H^i(\lambda)\cong H^{i-1}(R^1\texttt{ind}^{P(\alpha)}_B\lambda). $$
\item[{\tiny(5.c)}] If $ (\lambda,\alpha)=-1$, then $H^\bullet(\lambda)=0$.
\end{itemize}
\end{itemize}
\end{lemma}

\subsubsection{} Turn to the PID $A$ and its fractional field $\bbk$. In the remaining part of this section, we assume that $\text{char}(\bbk)=0$.
Then we have  $G_A$, $B_A$, $T_A$ and $P(\alpha)_A$. For a $B_A$-module $M$, we set $$H^i_{\alpha,A}(M)=R^i\texttt{ind}^{P(\alpha)_A}_{B_A}(M)$$ and $$H^i_A(\lambda):=R^i\texttt{ind}^{G_A}_{B_A}(A_\lambda) $$
  for natural $B_A$-module $A_\lambda$.
\begin{lemma}\label{lem: basic compu on sim ev root}
 The following statements hold.
\begin{itemize}
\item[(1)] If $(\lambda,\alpha)=-1$, then $H^\bullet_{\alpha,A}(\lambda)=0$.

\item[(2)] If $(\lambda,\alpha)\geq 0$, then $H^i_{\alpha,A}(\lambda)=0$ for all $i\ne 0$, and  $H^j_{\alpha,A}(r_\alpha.\lambda)=0$ for all $j\ne 1$. More precisely, for $s:=(\lambda,\alpha)\geq 0$.

\begin{itemize}
\item[{\tiny(2.a)}] $H^1_{\alpha,A}(r_\alpha.\lambda)\cong\sum_{i=0}^sAv_i'$
    is $A$-torsion free of rank $(r+1)$, including $A$-free basis elements $\{v'_i\mid i=0,1,\ldots,r\}$.
\item[{\tiny(2.b)}] $H^0_{\alpha,A}(\lambda)\cong\sum_{i=0}^sAv_i$ is $A$-torsion free of rank $(r+1)$, including $A$-free basis elements $\{v_i\mid i=0,1,\ldots,r\}$.
\item[{\tiny(2.c)}]  There exists $P(\alpha)_A$-module homomorphism
$$T_\alpha(r_\alpha.\lambda):  H^1_{\alpha,A}(r_\alpha.\lambda)\rightarrow H^0_{\alpha,A}(\lambda)$$
with $v'_i\mapsto {r\choose i}v_i$. And
$$T_\alpha(\lambda): H^0_{\alpha,A}(\lambda)\rightarrow H^1_{\alpha,A}(r_\alpha.\lambda)$$
with $v_i\mapsto (r-i)! i!v'_i$.
\end{itemize}
\item[(3)] Furthermore, if $(\lambda,\alpha)\geq 0$, then for any $i$, $$H^i_A(\lambda)\cong R^i\texttt{ind}^{G_A}_{P(\alpha)_A}(H^0_{\alpha,A}(\lambda))\cong H^i_{A}(H^0_{\alpha,A}(\lambda)).$$

\end{itemize}
\end{lemma}
\begin{proof} All statements follows from the previous lemma. In particular, (3) follows from its fourth statement.
\end{proof}
\subsubsection{}\label{sec: compose scalar}
More generally, for simple even root $\alpha$, and $\mu\in X(T)$ with $(\mu,\alpha)\geq 0$, we have  the following lemma.
\begin{lemma}\label{lem: furth comput evenro} $  H_A^i(r_\alpha.\mu)\cong R^{i-1}\texttt{ind}^{G_A}_{P(\alpha)_A}(H^1_{\alpha,A}(r_\alpha.\mu))\cong H^{i-1}_{A}(H^1_{\alpha,A}(r_\alpha.\mu))$.
\end{lemma}
\begin{proof} For supergroup schemes,  the following spectral sequence still holds as happening in the case of algebraic group schemes (see \cite[\S{I.4.5}]{Jan3})
$$R^i\texttt{ind}^{G_A}_{P(\alpha)_A}( R^j\texttt{ind}^{P(\alpha)_A}_{B_A} A_{r_\alpha.\mu})\Longrightarrow R^{i+j} \texttt{ind}^{G_A}_{B_A}(A_{ r_\alpha.\mu}).$$
On the other hand, by Lemma \ref{lem: basic compu on sim ev root}(2) we have $H^j_{\alpha,A}(r_\alpha.\mu)=0$ for all $j\ne 1$ when $(\mu,\alpha)\geq 0$.
Hence,  $H_A^i(r_\alpha.\mu)\cong R^{i-1}\texttt{ind}^{G_A}_{P(\alpha)_A}(H^1_{\alpha,A}(r_\alpha.\mu))\cong H^{i-1}_{A}(H^1_{\alpha,A}(r_\alpha.\mu))$.
\end{proof}
%
Furthermore, under the assumption $(\mu,\alpha)\geq 0$  there is some $j\in \bbn$ such that $H_A^j(\mu)$ is not torsion module. Then this $j$ is unique and $H^{j+1}(r_\alpha.\mu)$ is not a torsion module.
 By Lemmas \ref{lem: basic compu on sim ev root} and \ref{lem: furth comput evenro} we have the following homomorphisms
\begin{align}\label{eq: even comp over A}
\tilde T(r_\alpha.\mu): H^{j+1}_A(r_\alpha.\mu)\rightarrow H^j_A(\mu)
\end{align}
and
$$\tilde T(\mu): H^{j}_A(\mu)\rightarrow H^{j+1}_A(r_\alpha.\mu)$$
such that the composite of them are the multiplication by $(\mu,\alpha)!$.

\subsubsection{} Choose a reduced expression $w_\bz=r_{\alpha_N}r_{\alpha_{N-1}}\cdots r_{\alpha_1}$ for all $\alpha_i\in \Pi_\bz$ with $l(w_{\bz})=N:=|\Phi^+_\bz|$. For  $\lambda\in X^+(T)$  and  any $i$ we have the following homomorphism
$$\tilde T_{\alpha_i}(r_{\alpha_i}r_{\alpha_{i-1}}\cdots r_{\alpha_1}.\lambda):H^i_A(r_{\alpha_i}r_{\alpha_{i-1}}\cdots r_{\alpha_1}.\lambda)\rightarrow H^{i-1}_A(r_{\alpha_{i-1}}\cdots r_{\alpha_1}.\lambda).
 $$

 Set  $H^i_\bbk(\mu):= H^j_A(\mu)\otimes_A \bbk$.
 By composing  all these homomorphisms we have
 \begin{equation}\label{eq: even comp A}
 \widetilde T_{A,w_{\bar0}}: H_A^{l(w_\bz)}(w_{\bz}.\lambda)  \rightarrow H_A^0(\lambda).
  \end{equation}
Set $H_\bbk^i(\lambda)=H_A^i(\lambda)\otimes_A \bbk$. Using the arguments in \S\ref{sec: compose scalar} we further have the following observation.


\begin{lemma}\label{lem: K-iso}
 Suppose $\bbk$ is of characteristic $0$. Then the natural extension of $T_{A,w_\bz}$
  \begin{equation}\label{eq: even comp K}\widetilde T_{\bbk,w_{\bar0}}: H_\bbk^{l(w_\bz)}(w_{\bz}.\lambda)  \rightarrow H_\bbk^0(\lambda).
  \end{equation}
  is a $G_\bbk$-module isomorphism.
  \end{lemma}

\subsubsection{} Now we turn back to the base field $\tbk$.  The following result holds.
\begin{lemma}\label{lem: even comp nonzero} Suppose $\lambda\in X^+(T)$. Then the composed homomorphism
 \begin{equation}\label{eq: even comp over k}\widetilde T_{\tbk,w_{\bar0}}: V(\lambda)  \rightarrow H^0(\lambda)
  \end{equation}
  is nonzero.
\end{lemma}
\begin{proof} From Lemma \ref{lem: basic cohom comput} and the result on the composed homomorphism from  Weyl module $V_\ev(\lambda):=H^{l(w_{\bar 0})}_\ev(G/B, \mathscr{L}(\textbf k_{w_{\bar 0}.\lambda}))$ to $H^0_\ev(\lambda)$ is nonzero  (see \cite[\S{II}.6.16]{Jan3}). On the other side, by Theorem \ref{thm: Shi thm}(1), $V_\ev(\lambda)$ and $H^0_\ev(\lambda)$ are $G_\ev$-submodules of $V(\lambda)$ and $H^0(\lambda)$, respectively. So the lemma can be deduced from Theorem \ref{thm: Shi thm}(2).
\end{proof}

\begin{remark}  In general, the image of $\widetilde T_{\tbk,w_{\bar0}}$  contains $L(\lambda)$ as a proper submodule. It will be seen that both of them coincide  if and only if $\lambda$ is $p$-typical (see Theorem \ref{thm: p-typical}).
\end{remark}

\subsection{} Furthermore,  by (\ref{e6.5}) and (\ref{eq: even comp A}) we then have \begin{equation}\label{w01}
 \widetilde T_{A,\widehat w_\ell}:=\widetilde T_{A, \bar{w}_{\bo}}\circ \widetilde T_{A, w_{\bar 0}}: H_A^{l(w_{\bar 0})}(w_{\bar 0}.\lambda)\rightarrow   H^0_{\texttt{total},A}(\lambda)=H_A^0(G_A\slash B^{mn}, \lambda_{mn}).
 \end{equation}

\begin{prop}\label{c9.2} Let $\lambda\in X^+(T)$ be a typical weight. The following statements hold.
\begin{itemize}
\item[(1)] The composite $\widetilde T_{\textbf k,\widehat w_\ell}$ is a nonzero homomorphism which maps the   head of $V(\lambda)$ onto the socle of $\total=H^0(G/B^{(mn)},{\lambda_{mn}})$, both of which are  isomorphic to $L(-w_{\bar 0}\lambda+2\rho_{\bar 1})^*$.
    \item[(2)] The image $\widetilde T_{\textbf k,\widehat w_\ell}$ is exactly the simple socle of $\total$.
        \end{itemize}
\end{prop}

\begin{proof}
(1) By Lemma \ref{lem: even comp nonzero},  $T_{\tbk,w_{\bar0}}$ is nonzero.
Hence  $\texttt{im}(T_{\tbk,w_{\bar0}})^{{U^{(mn)}}^+}$ is nonzero, which by Lemma \ref{lem: U fixed}, coincides with $H^0(\lambda)^{{U^{(mn)}}^+}$. Furthermore, $H^0(\lambda)^{{U^{(mn)}}^+}$ is identical to the one-dimensional weight space $H^0(\lambda)_{\lambda_{mn}}$ . Combining with Proposition \ref{prop: odd comp nonzero} and its proof, we have that $\widetilde T_{\textbf k,\widehat w_\ell}$ must a nonzero homomorphism.

(2) Thanks to  Corollary \ref{coro: total socle}(2), the irreducible head of $V(\lambda)$  is isomorphic to $L(\gamma)^*$ which is exactly isomorphic to $L_{\bmn}(\lambda_{mn})=\texttt{soc}(\total)$.
By  Lemma \ref{lem9.1}(2), the ${B^{(mn)}}^+$-highest weight space is exactly $\total^{{B^{(mn)}}^+}$ which is one-dimensional, and generates $L_{\bmn}(\lambda_{mn})=\texttt{soc}(\total)$. Hence $L_{\bmn}$ is a multiplicity-one composition factor of $\total$. Set $M$ to be the image of $\widetilde T_{\textbf k,\widehat w_\ell}$. One side, $\texttt{soc}(M)=\texttt{soc}(\total)=L_{\bmn}(\lambda_{mn})$. On the other side, $M$ contains all composition factors of $M\slash \texttt{rad}(M)\cong V\slash \texttt{rad}(V(\lambda))\cong L_{\bmn}(\lambda_{mn})$. The multiplicity freeness  of $L_{\bmn}(\lambda_{mn})$ in the composition factors of $\total$ yields that $M$ is exactly $L_{\bmn}(\lambda_{mn})$.

The proof is completed.
\end{proof}

\section{Jantzen filtration of Weyl modules and sum formulas}

\subsection{General construction of Jantzen filtration}\label{sec: Gen Jan Fil}
In this subsection, we recall some general construction of Jantzen filtration for the readers's convenience. More details can be referred  to \cite[\S{II.8}]{Jan3}.


\subsubsection{} Keep the notations as in \S\ref{sec: conv-2} (and  as before as well). In particular, let $A$ be a principal ideal domain and $\bbk$ its filed fractional filed. Take $\frak{p}$ a maximal ideal of $A$. Let $\nu_{\frak p}$ be the ${\frak p}$-adic valuation of $\bbk$. If nonzero element $a\in A$, $\nu_{\frak p}(a)=r$ if and only if $a\in {\frak p}^r\backslash {\frak p}^{r+1}$.

For any $A$-module $M$ we can talk about the torsion submodule $M_{\text{tor}}$. Set $M_{\text{fr}}=M\slash M_{\text{tor}}$ the torsion free quotient. Then $M_{\text{fr}}$ is a projective $A$-module if $M$ is finitely generated. If $M$ is further a $G_A$-module, or to say, an object in the category $G_A\hmod$, we have $G_A$-submodule $M_{\text{tor}}$ and $G_A$-quotient module $M_{\text{fr}}$. If $M$ is finitely generated over $A$, then $\texttt{ch}(M_{\text{fr}})$ can be defined and coincides with  $\texttt{ch}(M\otimes_A \bbk)$.

In particular, with taking $A=\bbz$  we have  that $V_\bbz(\lambda)_{\text{fr}}\otimes_\bbz\tbk$ is isomorphic to a subquotient of $V(\lambda)$. Furthermore, for any commutative $\bbz$-algebra $A'$, we have $V_{A'}(\lambda)\cong V_\bbz(\lambda)\otimes_\bbz A'$.
In particular,   as $V(\lambda)\cong V_{\bbf_p}(\lambda)\otimes_{\bbf_p}\tbk$, we have
\begin{align}\label{eq: char base change-1}
\texttt{ch}V(\lambda)=\texttt{ch}V_{\bbf_p}(\lambda)=
\texttt{ch}V_{\bbz}(\lambda)_{\text{fr}}.
\end{align}

\subsubsection{}\label{8.4}
Set $\nu_\ppp(N)$ to be the length of the $A_\ppp$-module $N_\ppp=N\otimes_A A_\ppp$ for any $A$-module $N$.
 For $M\in G_{A}\hmod$, set
 \begin{equation}\nu_{\frak p}^c(M)=\sum\limits_{\mu}\nu_{\frak p}(M_\mu)e(\mu).
 \end{equation}

Let $M$ and $M'$ be torsion free $A$-modules in  $G_A\hmod$.   Consider a homomorphism  $\psi: M\rightarrow M'$ in $G_A\hmod$ which satisfies  $\psi\otimes \bbk:M\otimes_A \bbk\cong M'\otimes_A \bbk$.
Then we have  $\texttt{coker}(\psi)=M'/\psi(M)\in  G_{A}\hmod$ is a finitely generated $A$-torsion module.

Set \begin{equation}\label{s10.1'}\nu_{\frak p}(\psi)=\nu_{\frak p}(\texttt{coker}(\psi)) \hbox{ and } \nu^c_{\frak p}(\varphi)=\nu^c_{\frak p}(\texttt{coker}(\psi)).\end{equation}

Note that $\overline\psi: \overline M\rightarrow \overline M'$ is the homomorphism induced by $\varphi$ where $\overline M=M/\ppp M$ and $\overline M'=M'/\ppp M'$. Set
$$M^i:=\{m\in M|\psi(m)\in  \ppp^iM'; \forall i\in \mathbb N\}.$$
 Set $\overline M^i$ to be the image of $M^i$ in $\overline M$.
Then all $M^i$ are $G_A$-submodules of $M$ and all $\overline M^i$ are $G_{A\slash \ppp}$-submodule of $\overline M$.  Furthermore, we have  (cf. \cite[II.8.18]{Jan3})
\begin{align}\label{eq: image iso}
 \overline M\slash\overline M^1\cong \texttt{im}\overline\psi
\end{align}

\begin{equation}\label{s10.1}\overline M^1=\texttt{ker}\overline\psi.
\end{equation}
And
{ \begin{equation}\label{eq: filt sum}\sum\limits_{i>0}\texttt{ch}(\overline {M^i})=\nu^c_{\frak p}(\psi).
\end{equation}}

\subsubsection{}\label{sec: 8.13}
Suppose there  are two homomorphism $\varphi$ and $\varphi': M'\rightarrow M''$ in $G_A\hmod$ satisfying the above assumptions. Consider $\psi=\varphi'\circ\varphi: M\rightarrow M''$.

By the same arguments as in \cite[\S{II}.8.11]{Jan3} we have
 \begin{equation}\label{eq: coch sum}
 \nu^c_\frak p(\varphi'\circ\varphi)=\nu^c_\frak p(\varphi')+\nu^c_\frak p(\varphi).
 \end{equation}

In order to  make a distinction between different filtrations arising from different homomorphisms, we adopt an additional subscript like
$\{M^i_\varphi\}$ for the above filtration associated with $\varphi$. Then by  (\ref{eq: filt sum}) and (\ref{eq: coch sum})
 we have the following observation.

\begin{lemma}\label{lem: fil decomp}
$\sum\limits_{i>0}\texttt{ch}(\overline {M_\psi^i})=
\sum\limits_{i>0}\texttt{ch}(\overline {M_\varphi^i})+\sum\limits_{i>0}\texttt{ch}(\overline {{M'}_{\varphi'}^i})$.
\end{lemma}

\subsection{Arguments for the part arising from even reflections}

  From now on, we take $A=\bbz$, and take $\frak{p}=p\bbz$. Then $\bbf_p=A\slash \frak{p}$, and $\tbk=\bbf_p\otimes_\bbz\tbk$.

  Set $\varphi:=\widetilde T_{A,w_{\bar0}}$ in (\ref{eq: even comp A}). We already have the following homomorphism
 \begin{equation*}
 \varphi: H_A^{l(w_\bz)}(w_{\bz}.\lambda)  \rightarrow H_A^0(\lambda).
  \end{equation*}
  By Lemmas \ref{lem: K-iso} and \ref{lem: even comp nonzero}, the construction in \S\ref{sec: Gen Jan Fil}  can be applied.

  \begin{prop}\label{prop: Jan Fil Even}   Let $\lambda\in X^+(T)$.
There is a filtration of $G$-modules
$$V(\lambda)=V_\varphi(\lambda)^0\supset V_\varphi(\lambda)^1\supset\cdots $$ such that the following sum formula holds
\begin{align*}
&\sum\limits_{i>0}\texttt{ch} V_\varphi(\lambda)^i
=\sum\limits_{\alpha\in \Phi^+_{\bz} }\sum\limits_{0<mp<(\lambda+\rho_{\bar0},\alpha)}
\nu_p(mp)\chi (r_{\alpha,mp}.\lambda)
\end{align*}
where $\nu_p(mp)$ means the $p$-adic valuation of $mp$.
\end{prop}

\subsection{Arguments for the part arising from odd reflections} Set $\varphi'=\widetilde T_{A,{w}_{\bar 1}}$ in (\ref{e6.5}). We already have the following homomorphism
$$ \varphi': H^0_A(\lambda) \rightarrow H^0_{\texttt{total},A}(\lambda).$$
Lemma \ref{lem: K-iso Odd} and Proposition \ref{prop: odd comp nonzero} ensure  the construction in \ref{sec: Gen Jan Fil}.
By (\ref{eq: filt sum}) and (\ref{eq: coch sum}) again, we have
\begin{align}\label{eq: fil odd sum first}
\sum_{i>0}\texttt{ch} V_{\varphi'}(\lambda)^i=\sum_{i} \nu_p^c(\texttt{coker}(\tilde T_{\bk,\beta_i}).
\end{align}
Note that by Lemma \ref{L4.1}(1), $\texttt{coker}(\tilde T_{\bk,\beta_i})=0$ when $(\lambda,\beta_i)\not\equiv 0\mod p$. By Lemma \ref{L4.1}(2), Corollary \ref{cor: odd grod pre} along with \S\ref{sec: character}, Equation (\ref{eq: fil odd sum first}) gives rise to the following formula.

\begin{align}\label{eq: Jan Fil Odd}
\sum_{i>0}\texttt{ch} V_{\varphi'}(\lambda)^i
=&\sum\limits_{\beta_i\in \Phi^+_{\bo};\; p\mid(\lambda_{i-1},\beta_i)}
(\texttt{ch}(H^0(G\slash K_{\beta_i}, \lambda_{i-1}))+\cr
&+\sum_{k=1}^\infty(-1)^k \texttt{ch}(H^0(G\slash K_{-\beta_i}, \lambda_{i-1}+k\beta_i))).
\end{align}

Set $\scrW(\lambda):=\texttt{ch}(H^0_\ev(\lambda))$ which is computed via the Weyl character formula. Denote
$\Xi_{i}=\Pi_{\beta\in (\Phi^+_{\beta_i})_\bo}(1+e^{-\beta})$.
Note that $K_{\mp\beta_i}^+$ corresponds to the positive root systems $\Phi^+_{\beta_{i-1}}$ and $\Phi^+_{\beta_i}$ respectively. Here $\Phi^+_{\beta_{i-1}}$ when $i=1$ just stands for the standard positive root system $\Phi^+$.
Keep it in mind that $(\lambda_{i-1},\beta_i)\equiv0\mod p$ is equivalent to $(\lambda+\rho,\beta_i)\equiv0\mod p$.
 Theorem \ref{thm: Shi thm}(2) and (\ref{eq: Jan Fil Odd}) give rise to the following result.
\begin{prop}\label{prop: Jan Fil Odd} Keep the notations and assumptions as above. Then the following formula holds
\begin{align*}
\sum_{i>0}\texttt{ch} V_{\varphi'}(\lambda)^i
=\sum\limits_{\beta_i\in \Phi^+_{\bar 1};\; p\mid(\lambda+\rho,\beta_i)}
(\Xi_i\scrW(\lambda_{i-1})+\Xi_{i-1}\sum_{k>0}(-1)^k\scrW(\lambda_{i-1}+k\beta_i)).
\end{align*}
Here we appoint that $\Xi_0$ is just $\Xi$ in Proposition \ref{lem: totally ind ch}.
\end{prop}

\subsection{Total arguments} We are in a position to introduce our main result.
Set $\psi:=  \widetilde T_{A,\hat w_{\bar 0}}$   in (\ref{w01}). Then we already have
 $$\psi: V_A(\lambda)\rightarrow   H^0_{\texttt{total}, A}(\lambda)$$
 with $\psi=\varphi'\circ\varphi$.
By Lemma  \ref{lem: fil decomp}, along with Propositions \ref{prop: Jan Fil Even} and \ref{prop: Jan Fil Odd} we have

\begin{theorem}\label{main thm}Keep the notations as above. Let $\lambda\in X^+(T)$ be a typical weight.
There is a filtration of $G$-modules
$$V(\lambda)=V_\psi(\lambda)^0\supset V_\psi(\lambda)^1\supset\cdots $$ such that

\begin{itemize}
\item[(1)] $V(\lambda)/V_\psi(\lambda)^1\cong L(-w_{\bar 0}\lambda+2\rho_{\bar 1})^*$.
\item[(2)] The following sum formula holds
\begin{align*}
\sum\limits_{i>0}\texttt{ch} V_\psi(\lambda)^i
=&\sum\limits_{\alpha\in \Phi^+_{\bar 0} }\sum\limits_{0<mp<(\lambda+\rho_{\bar0},\alpha^{\vee})}
\nu_p(mp)\chi (r_{\alpha,mp}.\lambda)+\cr
+&\sum\limits_{\beta_i\in \Phi^+_{\bar 1};\; p\mid(\lambda+\rho,\beta_i)}
(\Xi_i\scrW(\lambda_{i-1})+\Xi_{i-1}\sum_{k>0}(-1)^k\scrW(\lambda_{i-1}+k\beta_i)).
\end{align*}
\end{itemize}
\end{theorem}

\begin{proof} The first statement follows from Lemma \ref{lem8.2}, Corollary \ref{c9.2} and (\ref{eq: image iso}) along with Proposition \ref{c9.2}(2).
 The second one follows from the above two propositions.
\end{proof}

\begin{remark}\label{rem: p-typical} When $\lambda$ is $p$-typical, i.e.
$(\lambda+\rho,\beta_i)\not\equiv 0 \mod p$ for $1\leq i \leq mn$, we have $$\sum\limits_{i>0}\texttt{ch} V_\psi(\lambda)^i=\sum\limits_{\alpha\in \Phi^+_{\bar 0} }\sum\limits_{0<mp<(\lambda+\rho_{\bar0},\alpha)}\nu_p(mp)\chi (r_{\alpha,mp}.\lambda),$$ and  $V(\lambda)/V_\psi(\lambda)^1\cong L(\lambda)$ (cf. Remark \ref{r8.3}). In this case, the statement  is the same as in the case of reductive algebraic groups (see \cite[\S II.8.19]{Jan3}).
\end{remark}

\section{Kac modules and realizations of $p$-typical irreducible modules}
\subsection{Kac modules}\label{sec: kac mod}    Let $L_\ev(\lambda)$ denote the socle of $H^0_\ev(\lambda)$.     {Considering the closed subgroup scheme $B^+G_\ev$ of $G$, associated which $B^+G_\ev(R)$ for $R\in\sak$ is a parabolic subgroup of $G(R)$ consisting of matrices whose lower odd block submatrix is zero.
There is a natural epimorphism of supergroups $B^+G_\ev\to G_\ev$.}
  We define a $\texttt{Dist}(G)$-module $\texttt{Dist}(G)\otimes_{\texttt{Dist}(B^+G_\ev)} L_\ev(\lambda)$ with trivial $\texttt{Dist}(U^+)_\bo$-action on $L_\ev(\lambda)$, which is called a Kac module. We  denote it by  $\mathscr{K}(\lambda)$. Here $\texttt{Dist}(\diamondsuit)$ stands for the distribution algebra of a (super)group scheme $\diamondsuit$ (see \cite{Jan3}, or \cite{BKu} for  details).
Recall that the category of finite-dimensional rational $G$-modules  is equivalent to the category of finite-dimensional integrable $\texttt{Dist}(G)$-modules  (see \cite[Corollary 3.5]{BKu}). Here by an integrable $\texttt{Dist}(G)$-module $M$ it means there is a $T$-module structure on $M$ compatible with $\texttt{Dist}(G)$-module structure (see
$\texttt{Dist}(G)$-$T$-modules for algebraic groups in \cite[Page 171]{Jan3}).
So both will be identified in the following.
\begin{lemma}\label{lem: kac mod} There is a nontrivial homomorphism from $\mathscr{K}(\lambda)$ onto $L(-w_\bz\lambda)^*$ for any $\lambda\in X^+(T)$. This homomorphism becomes an isomorphism if and only if $ L(-w_\bz\lambda)^*$ has weight space of $w_\bz\lambda-2\rho_\bo$.
\end{lemma}
\begin{proof}  It is readily seen that $G$ has a closed subgroup scheme $\bmn G_\ev$. Now we  exploit the structure theory of algebraic supergroups of Chevalley type (see \cite[\S5]{FG12}).
     {As mentioned in the proof of Theorem \ref{thm: Shi thm}(1) where we have already a normal subgroup scheme $G_1$ of $G$}, it can be shown that both $\bmn G_\ev$ and $\bmn$ have a normal subgroup scheme $L_1$ such that  $$\bmn G_\ev\slash L_1\cong G_\ev \;\text{ and }\bmn\slash L_1\cong B_\ev.$$
 Such a normal subgroup scheme $L_1$ can be described as $L_1(R)=\bmn(R)\cap G_1(R)$ for any $R\in \sak$, where $G_1(R)$  is the same as in the proof of Theorem \ref{thm: Shi thm}(1).
 By the same arguments as in the proof of  \cite[Lemma 10.4]{Z1},
  a super analogue of the result \cite[Proposition 6.11]{Jan3} yields the following isomorphism of $\bmn G_\ev$-modules
  $$H^0(\bmn G_{\ev}\slash \bmn,\lambda)\cong H^0(G_{\ev}\slash B_\ev,\lambda).$$

 Recall that $H^0(G\slash \bmn,\lambda)$ is finite-dimensional (see Theorem \ref{lem8.3}), and it  has a $\bmn G_\ev$-submodule $H^0(\bmn G_{\ev}\slash \bmn,\lambda)\cong H^0(G_{\ev}\slash B_\ev,\lambda)$ which can be regarded a $\texttt{Dist}(B^{(mn)}G_\ev)$-module with $\texttt{Dist}(\bmn)_\bo$-trivial action. Note that $\texttt{Dist}(U^{(mn)})_\bo=\texttt{Dist}(U^+)_\bo$. By a classical result of reductive algebraic groups, $H^0(G_{\ev}\slash B_\ev,\lambda)$ has simple socle $L_\ev(\lambda)$. So the socle $L_\ev(\lambda)$ of $H^0(\bmn G_{\ev}\slash \bmn,\lambda)$ is actually an irreducible $\texttt{Dist}(\bmn G_\ev)$-module with trivial $\texttt{Dist}(U^+)_\bo$-action.
 %
  This one is actually an irreducible $\texttt{Dist}(B^+G_\ev)$-irreducible module with trivial $\texttt{Dist}(U^+)_\bo$-action.
On the other side,  this $L_\ev(\lambda)$ has one-dimensional $B_\ev^+$-highest weight space of weight $\lambda$. By the same arguments as in the proof of Lemma \ref{lem: U fixed}, it is known that there is one-dimensional subspace
$$H^0(G\slash \bmn, \lambda)_\lambda=L_\bmn(\lambda)_\lambda=H^0(G\slash\bmn,\lambda)^{\bmn^+}.$$ Hence the $G$-submodule generated by this  $L_\ev(\lambda)$ in $H^0(G\slash \bmn,\lambda)$
is exactly $L_\bmn(\lambda)$. Note that $\texttt{Dist}(G)=\texttt{Dist}(U^-)_\bo\texttt{Dist}(B^+G_\ev)$.
This means, $L_\bmn(\lambda)=\texttt{Dist}(U^-)_\bo L_\ev(\lambda)$.
 By the universality of tensor products, there is a nontrivial homomorphism from $\mathscr{K}(\lambda)$ onto $L_\bmn(\lambda)$ sending $1\otimes L_{\ev}(\lambda)$ to the socle of $H^0(\bmn G_{\ev}\slash \bmn,\lambda)$.  Keep it in mind that $L_\bmn(\lambda)\cong L(-w_\bz\lambda)^*$ (see Lemma \ref{lem9.1}).
 The first part is proved.

 For the second part, consider   $\ggg=\texttt{Lie}(G)$ which has a direct-sum decomposition of root spaces $\hhh+\sum_{\alpha\in \Phi}\ggg_\alpha$ with $\hhh=\texttt{Lie}(T)$. By the routine arguments for highest weight categories, $\mathscr{K}(\lambda)$ has a unique maximal submodule $N$ which is the direct sum of all proper submodules. Note that as a vector space, $\mathscr{K}(\lambda)=\texttt{Dist}(U^-)_\bo\otimes L_\ev(\lambda)$. The latter is exactly equal to  $\bigwedge^\bullet(\sum_{\beta\in \Phi^+_\bo}\ggg_{-\beta})\otimes L_\ev(\lambda)$ (see Lemma \cite[Lemma 3.1]{BKu}).  If $N$ is nonzero, by some trivial but a little tedious arguments $N$ must contain one-dimensional $B^+$-lowest weight space $\mathscr{K}(\lambda)_{w_\bz\lambda-2\rho_\bo}$. In this case, the irreducible quotient of $\mathscr{K}(\lambda)$ does not contain nonzero weight space of weight  $\lambda_{mn}$.
 Hence $\mathscr{K}(\lambda)$ is irreducible if and only if $L(\lambda)$ does not contain nonzero weigh space of weight $w_\bz(\lambda)-2\rho_\bo$. The second statement  follows.
\end{proof}

\begin{remark}
Note that  $\bmn=w_\bo(B)$ for $G=\GL(m|n)$.
It is not hard to see for $G=\GL(m|n)$ or $G=\text{OSp}(m|2n)$, $w_\bo(B) G_\ev=B^+ G_\ev$, and   $\mathscr{K}\cong\texttt{ind}^G_{B^+G_\ev}L_\ev(\lambda)$ by using the structural properties of $\texttt{Dist}(G)$ (see \cite{BKu}, \cite{SW}) and of algebraic supergroups of  Chevalley type (see \cite[\S5.3]{FG12}).
So one can expect an alternative definition  of a Kac module $\mathscr{K}(\lambda)$ for basic classical supergroups via induced modules. 
\end{remark}

\subsection{An application}

    {As an application, we give an alternative proof of  known results (see \cite[Propostion 12.10]{Z1}, and \cite[Theorem 1]{Ma}).}
\begin{theorem}\label{thm: p-typical} The following statements are equivalent for $G=\GL(m|n)$ over $\tbk$.
\begin{itemize}
\item[(1)]  $\lambda\in X^+(T)$ is  $p$-typical.
 \item[(2)] $L(\lambda)\cong L(-w_\bz\lambda+2\rho_\bo)^*$.
 \item[(3)] The irreducible module $L(\lambda)$ is isomorphic to  $\mathscr{K}(\lambda-2\rho_\bo)$.
\end{itemize}
\end{theorem}

\begin{proof} Note that $(1)$ is equivalent to say $L(\lambda)=L_{B^{(mn)}}(\lambda_{mn})$. The latter is equivalent to say $L(\lambda)$ has one-dimensional $B^+$-lowest weight space of weight $\lambda_{mn}$.  So the equivalence of the first two follows from  Lemma \ref{L4.1} and Remark \ref{r8.3}.

By the above arguments, the equivalence of the second and third ones is deduced from Lemma \ref{lem: kac mod}.
\end{proof}

\begin{remark}\label{rem: p typical realization} This theorem can be regarded a modular version of the result on Kac module realization of typical irreducible modules over complex numbers (see Remark \ref{rem: complex typical}).
\end{remark}

\subsection*{Acknowledgement} The authors express deep thanks to the anonymous referee for helpful comments making them improve an old version of the manuscript.

\end{document}